\documentclass{amsart}

\usepackage[english]{babel}
\usepackage[utf8]{inputenc}
\usepackage{bbm}
\usepackage{amssymb}
\usepackage[usenames,dvipsnames,svgnames,table]{xcolor}
\usepackage{paralist}
\usepackage{cleveref}
\usepackage{memhfixc}
\usepackage{graphicx}
\usepackage{comment}
\usepackage{mathtools}
\usepackage{pstricks,pst-plot,pstricks-add}

\definecolor{midnight_blue}{RGB}{25,25,112}
\definecolor{navy}{RGB}{0,0,128}
\definecolor{corn_blue}{RGB}{100,149,237}
\definecolor{dark_sea_green}{RGB}{143,188,143}
\definecolor{medium_sea_green}{RGB}{60,179,113}
\definecolor{sea_green}{RGB}{46,139,87}
\definecolor{light_sea_green}{RGB}{32,178,170}
\definecolor{sienna}{RGB}{160,82,45}
\definecolor{tan}{RGB}{210,180,140}
\definecolor{sky_blue}{RGB}{135,206,250}
\definecolor{pale_green}{RGB}{152,251,152}
\definecolor{royal_blue}{RGB}{65,105,225}
\definecolor{tomato}{RGB}{255,99,71}

\newtheorem{theorem}{Theorem}[section]
\newtheorem{lemma}[theorem]{Lemma}
\newtheorem{proposition}{Proposition}[section]

\theoremstyle{definition}
\newtheorem{definition}[theorem]{Definition}

\theoremstyle{remark}
\newtheorem{remark}[theorem]{Remark}

\numberwithin{equation}{section}

\begin{document}

\pdfoutput=1

\title{HARNACK INEQUALITIES FOR SYMMETRIC STABLE L\'EVY PROCESSES}

\author{MARINA SERTIC}
\address{Fakultät für Mathematik, Universität Bielefeld, Universitätsstraße 25, 33615 Bielefeld}
\email{msertic@math.uni-bielefeld.de}
\thanks{The author appreciates the support of the International Graduate College "Stochastics and Real World Models", Universität Bielefeld.}



\date{}



\begin{abstract}
In this paper we consider Harnack inequalities with respect to a symmetric $\alpha$-stable L\'evy process $X$ in $\mathbb{R}^d$, $\alpha \in (0,2)$, $d\geq 2$.
We study the example from the article \cite{bg-sz-1}. There, the authors have associated the Harnack inequality with the relative Kato condition, which is a condition on the L\'evy measure.
By checking the condition, in the case $\alpha \in (0,1)$, they have established that the Harnack inequality does not hold.
We give an alternative proof of this fact, using the setting of \cite{bg-sz-1}. We define the harmonic functions explicitly.
For a given starting point of the process, we examine the probability of hitting a certain set at the first exit time of a unit ball.

Moreover, we also examine the weak Harnack inequality for a certain class of symmetric $\alpha$-stable L\'evy processes. We consider a symmetric $\alpha$-stable L\'evy process, $\alpha \in (0,2)$, for which a spherical
part $\mu$ of the L\'evy measure is a spectral measure. In addition, we assume that $\mu$ is absolutely continuous with respect to the uniform measure $\sigma$ on the sphere 
and impose certain bounds on the corresponding density. Eventually, we show that the weak Harnack inequality holds.
\end{abstract}

\maketitle



\subsection*{Acknowledgment} 
The author is thankful to Moritz Kassmann and Mateusz Kwa\'snicki for valuable discussions.

\hfill \break

\section{Introduction}
In the paper, we use the notation
\begin{align*}
 B(x_0,r)=\{x : |x-x_0| \leq r\}, r\geq 0,
\end{align*}
and
\begin{align*}
 B_{r}=B(0,r), r\geq 0.
\end{align*}

We consider \emph{a symmetric $\alpha$-stable L\'evy process},
which has the characteristic function of the form
\begin{align}\label{k_fun}
 \mathbb{E}^{0}\big[e^{i u\cdot X_t}\big]=e^{-t\Phi(u)}, u \in \mathbb{R}^d,\ t\geq 0,
\end{align}
where the characteristic exponent $\Phi$ is given by
\begin{align}\label{k_e}
\Phi(u)=\int_{\mathbb{S}^{d-1}} |u\cdot \xi|^\alpha \mu (\mathrm{d}\xi).
\end{align}
The measure $\mu$ is symmetric, finite and non-zero on $\mathbb{S}^{d-1}$ (see \cite{sato}, Theorem 14.13). 
Le the measure $\mu$ be absolutely continuous with respect to the uniform surface measure on $\mathbb{S}^{d-1}$ and denote its density by $f_\mu$.

\emph{The potential density (or the heat kernel) $p(t,x,y)=p(t,y-x)$} is determined by the Fourier transform 
\begin{align*}
 \int_{\mathbb{R}^d} e^{i\xi\cdot x}\ p(t,x)\ \mathrm{d}x=e^{-t\Phi(\xi)},\ \xi \in \mathbb{R}^d,\ t \geq 0.
\end{align*}

\begin{definition}
\emph{The Green function (or the potential kernel)} is defined by
\begin{align}\label{Gr_f}
 G(x,y)=\int_{0}^\infty p(t,x,y)\ \mathrm{d}t,\ x,y \in \mathbb{R}^d.
\end{align}
\end{definition}

\begin{definition}
Let $D$ be an open set, $D\subset \mathbb{R}^d$. \emph{The Green function of $X^D$} is defined by
\begin{align}\label{Gr_f_D}
G_D(x,y)=G(x,y)-\mathbb{E}^x[G(X_{\tau_D}, y)],\ x,y\in D.
\end{align}
\end{definition}

\hfill \break

\section{Harnack Inequality}
\begin{definition}
A non-negative Borel measurable function $u:\mathbb{R}^d \to \mathbb{R}$ is \emph{harmonic in an open set $D\subset \mathbb{R}^d$} with respect to the process $X$ if
\begin{align*}
 u(x)=\mathbb{E}^{x}[u(X_{\tau_U})],\ x\in U,
\end{align*}
for every bounded open set $U$ such that $\bar{U}\subset D$. Here, $\tau_U=\inf\{t>0 : X_t \notin U\}$ is the first exit time from the set $U$.
\end{definition}

\begin{definition}
\emph{The Harnack inequality} for a symmetric $\alpha$-stable L\'evy processes $X$ holds true if there is a constant $K\geq 1$ such that for every non-negative function $u\colon \mathbb{R}^d\rightarrow \mathbb{R}$
which is harmonic in $B_1$ with respect to $X$, the inequality
\begin{align*}
u(x)\leq K\cdot u(y),\ x,y \in B_{1/2}.
\end{align*}
holds true. 
\end{definition}

\subsection{Construction of the Sequence of Harmonic Functions}
\begin{lemma}{\label{lema}}
Let $c>1$.
There are sequences $(\alpha_n)$ and $(\beta_n)$ of positive numbers such that:
\begin{align}
&A=\sum\limits_{n=1}^\infty \alpha_n<\frac{\pi}{2},\ \ B=\sum\limits_{n=1}^\infty \beta_n<\frac{\pi}{2},\\
&\lim_{n \to \infty} (A_n+B_n)=A+B=\frac{\pi}{2},\\
&\lim_{n \to \infty}\frac{\tan(A_n + B_{n+1})}{\tan(A_n + B_n)} = +\infty \label{uvjet},\\
&(c+2)\cdot \tan(A_n + B_{n+1}) \leq (c-1)\cdot \tan(A_{n+1} + B_{n+1}),\ n \in \mathbb{N}\label{uvjet_d},
\end{align}
where $A_n$ and $B_n$ denote the
partial sum of the sequences $(\alpha_n)$ and $(\beta_n)$, respectively, i.e. $A_n= \sum\limits_{k=1}^n \alpha_k$, $B_n= \sum\limits_{k=1}^n \beta_k$.
\end{lemma}

\begin{proof}
We construct inductively sequences $(\alpha_n)$ and $(\beta_n)$ of positive numbers which fulfill the conditions of the lemma.
Set $K=(c+2)/(c-1)$.

Choose $\alpha_1$ and $\beta_1$ so that
\begin{align*}
 \alpha_1 + \beta_1 <  \pi/2.
\end{align*}
Next, choose $\beta_2$ and $\alpha_2$ so that:
\begin{align*}
&\beta_2 < \pi/2 - (A_1 + B_1), \\
&\arctan\big(K\cdot \tan(A_1+B_2)\big) - (A_1 + B_2) <\alpha_2 < \pi/2 - (A_1 + B_2).
\end{align*}
Assume $\beta_1, \ldots, \beta_n$, $\alpha_1, \ldots, \alpha_n$ are
chosen. Choose $\beta_{n+1}$ and $\alpha_{n+1}$ so that:
\begin{align*}
&\beta_{n+1} < \pi/2 - (A_n + B_n), \\
&n+1 \leq \frac{\tan(A_n + B_{n+1})}{\tan(A_n + B_n)},\\
&\arctan\big(K\cdot \tan(A_n+B_{n+1})\big) - (A_n + B_{n+1}) < \alpha_{n+1}\\
&\alpha_{n+1}< \pi/2 - (A_n + B_{n+1}).
\end{align*}

Notice that the choice of the sequence $(\beta_n)$ is possible, since
\begin{align*}
 \lim_{h \nearrow \pi/2 - c} \frac{\tan(h+c)}{\tan(c)} = +\infty,
\end{align*}
for fixed $c \in (0,\frac{\pi}{2})$.

Furthermore, the choice of of the sequence $(\alpha_n)$ is possible due to the choice of $(\beta_{n})$.

Now, from the choice of $(\alpha_n)$, we have
\begin{align*}
 \arctan\big(K\cdot \tan(A_n+B_{n+1})\big) <A_{n+1} + B_{n+1} < \pi/2,
\end{align*}
which, together with the choice of $(\beta_n)$, implies
\begin{align*}
 \lim_{n \to \infty} (A_n+B_n)=\frac{\pi}{2}.
\end{align*}
\end{proof}


\begin{remark}
At this point, we would like to mention that none of the following sequences satisfies the property \eqref{uvjet}
from the Lemma \ref{lema}.

\begin{enumerate}
\item $\alpha_k = 2^{-k}, \ \beta_k = \frac{1}{k(k+1)}$
\item $\alpha_k = a^{-k}, \ \beta_k = b^{-k}$, \  $1 < b \leq a$,
\item $\alpha_k =  \frac{1}{k(k+1)} , \ \beta_k = k^{-1-\delta},\ \delta > 0$, 
\item $\alpha_k = 2^{-(k^3)}, \ \beta_k = 2^{-(k^2)}$,\\
      $\alpha_k = a^{-(k^3)}, \ \beta_k = b^{-(k^2)}$,\quad $a>1,b>1$
\end{enumerate}
The details can be found in Lemma \ref{lema_pr} (Appendix \ref{Ap}). 
\end{remark}

\begin{definition}
Let $\delta, x \in (0,1)$ and $(\alpha_n)$ and $(\beta_n)$ be sequences as in
Lemma \ref{lema}. We define the sets
\begin{align*}
S_n^{\delta}(x)&=\{(\delta,y):A_{n-1}+B_n<\arctan\Big(\frac{y}{x+\delta}\Big)<A_n+B_n\},\\
L_n^{\delta}(x)&=\bigcup_{z\in [x,1]} S_n^{\delta}(z).
\end{align*}
\end{definition}

\begin{remark}
Notice that, by the definitions of $S_n^{\delta}(x)$ and $L_n^{\delta}(x)$, we have
\begin{align}\label{L_n}
L_n^{\delta}(x)=\{(\delta,y):(x+\delta)&\cdot\tan(A_{n-1}+B_n)<y<(1+\delta)\tan(A_n+B_n)\}.
\end{align}
\end{remark}

The sets $S_n^{\delta}(x)$ and $L_n^{\delta}(x)$ are illustrated at the pictures that follow.


\includegraphics[scale=1]{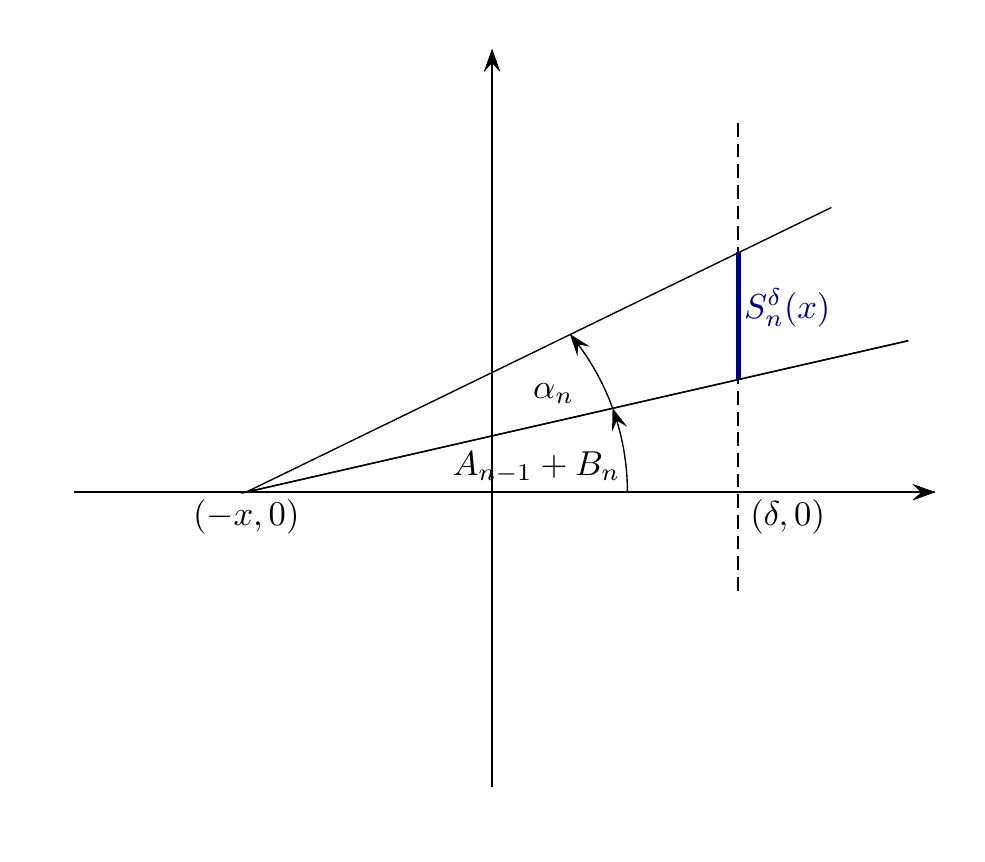}

\includegraphics[scale=1]{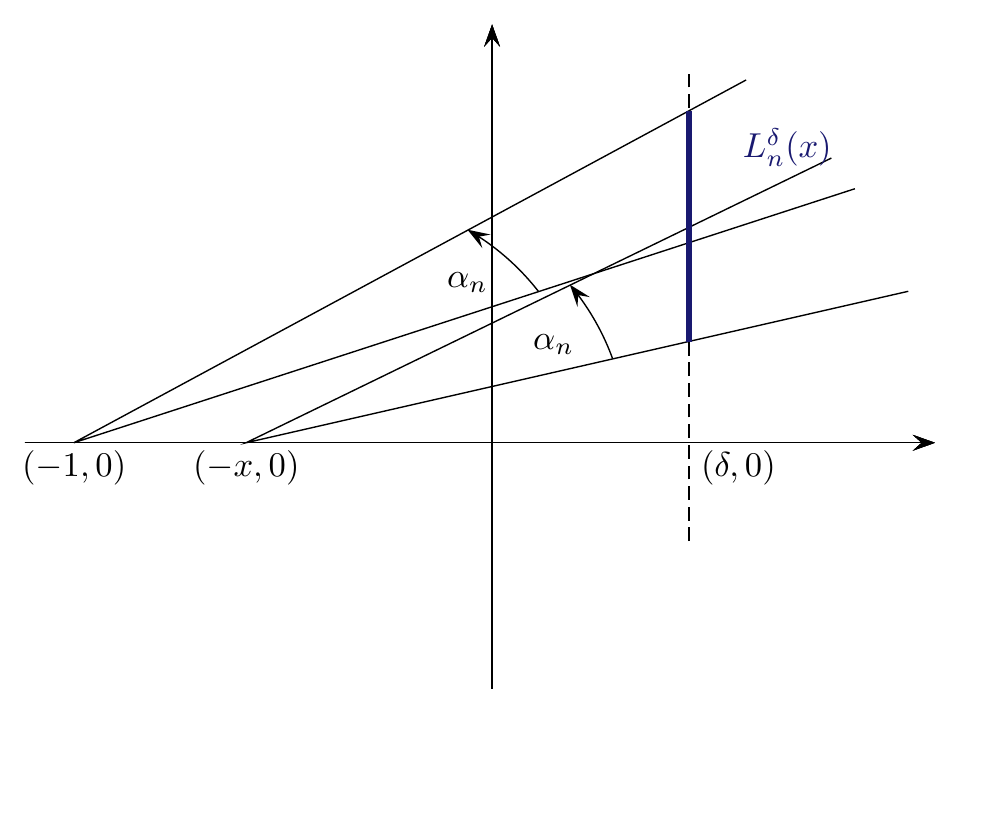}

\begin{remark}
$S_n^{\delta}(x)$ describes the set on the line $\{(x,y): x=\delta, y>0\}$,
which is seen from the point $(-x,0)$ by the cone $K_n$, and
$L_n^{\delta}(x)$ describes the set on the line $\{(x,y):
x=\delta, y>0\}$ which is seen from the set $[-1,-x]\times \{0\}$ by the
cone $K_n$.
\end{remark}

\begin{proposition}
Let $(\alpha_n)$ and $(\beta_n)$ be sequences as in Lemma \ref{lema}.
For every $\delta\in (0,1)$ and every $x \in (0,1)$ there is
$n_0=n_0(\delta,x)$ such that for every $n\geq n_0$ 
$$
L_n^{\delta}(x)\cap L_{n+1}^{\delta}(x)=\emptyset.
$$
\end{proposition}

\begin{proof}
Let $\delta, x\in (0,1)$.
Recall, by \eqref{L_n}, we have
\begin{align*}
L_n^{\delta}(x)=\{(\delta,y):(x+\delta)&\cdot\tan(A_{n-1}+B_n)<y<(1+\delta)\tan(A_n+B_n)\}.
\end{align*}
Now, from Lemma \ref{lema}, condition \eqref{uvjet}, it follows that there is $n_0=n_0(\delta,x)$ such that for every $n\geq n_0$,
\begin{align*}
L_n^{\delta}(x)\cap
L_{n+1}^{\delta}(x)=\emptyset,
\end{align*}
which proves the proposition.
\end{proof}

\begin{definition}
Let $\delta \in(0,1)$ and $(\alpha_n)$ and $(\beta_n)$ be sequences as in
Lemma \ref{lema}. For $z=(z_1, z_2), y=(y_1, y_2)\in \mathbb{R}^2$, we
define the sets
\begin{align*}
\tilde{S}_n^{\delta}(z)&=\{(\delta,u):A_{n-1}+B_n<\arctan\Big(\frac{u-z_2}{\delta+z_1}\Big)<A_n+B_n\},\\
\tilde{L}_n^{\delta}(y)&=\bigcup_{\{z:z_1\in [y_1,1], z_2=y_2\}}\tilde{S}_n^{\delta}(z).
\end{align*}
\end{definition}

\begin{remark}
$\tilde{S}_n^{\delta}(z_1,z_2)$ describes the set on the line $\{(x,u): x=\delta, u>0\}$,
which is seen from the point $(-z_1,z_2)$ by the cone $K_n$, and
$\tilde{L}_n^{\delta}(y_1,y_2)$ describes the set on the line $\{(x,u):
x=\delta, u>0\}$, which is seen from the set $[-1,-y_1]\times \{y_2\}$ by the
cone $K_n$.
\end{remark}

\begin{remark}
Let $w=(w_1,w_2)\in\mathbb{R}^2$. Notice that by the definitions of $\tilde{S}_n^{\delta}(z)$ and $\tilde{L}_n^{\delta}(y)$, we have
\begin{equation}\label{til_L_n}
\begin{aligned}
\tilde{L}_n^{\delta}(w)=\{(\delta,u): &(w_1+\delta)\cdot\tan(A_{n-1}+B_n)+w_2 <\\
				      &<u<(1+\delta)\cdot\tan(A_n+B_n)+w_2\}.
\end{aligned}
\end{equation}
\end{remark}

By means of the two propositions below, we will define the sequence of harmonic functions so that the Harnack inequality does not hold. To be more precise, we will construct the sequence of sets $(B_n)$ and consequently
the sequence of functions $(u_n)$,
so that
\begin{align*}
 \lim_{n\rightarrow \infty}\frac{u_n(1/2,0)}{u_n(0,0)}=0.
\end{align*}

\includegraphics[scale=1]{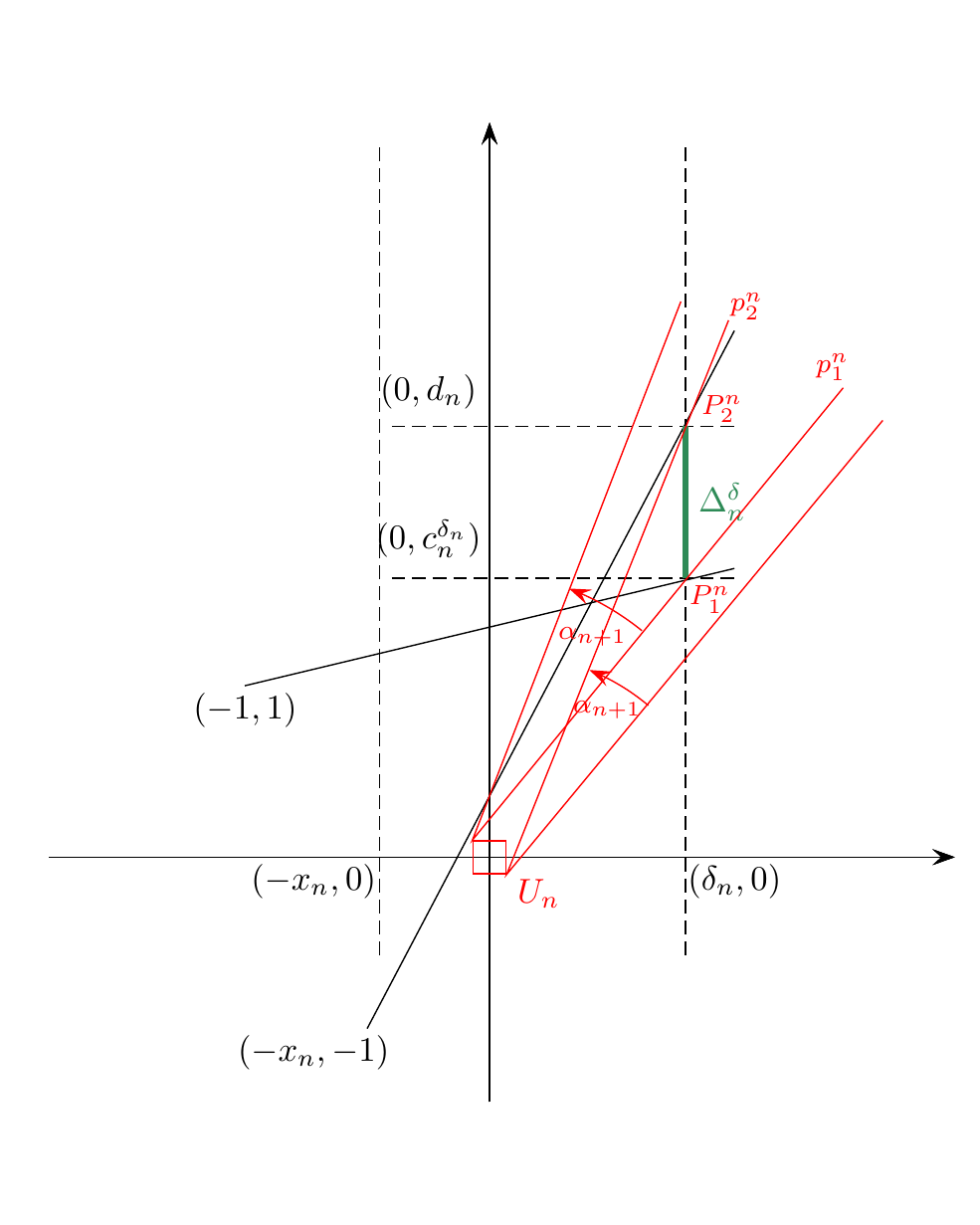}

\begin{proposition}{\label{prop4}}
Let $(\alpha_n)$, $(\beta_n)$ be sequences as in Lemma \ref{lema}.
There are sequences $(\delta_n)$, $(x_n)$, $(y_n)$, $(c_n)$ and $(d_n)$ of positive numbers such that
\begin{align*}
&V_n\subset \tilde{S}_{n+1}^{\delta_n}(-x_1,x_2)\cap \Delta_n^{\delta_n},
\end{align*}
for every $x=(x_1, x_2)\in U_n$,
where
\begin{align*}
U_n&\coloneqq (-y_n, y_n)^2,\\
V_n&\coloneqq \{\delta_n\}\times(c_n,d_n).
\end{align*}

\begin{remark}
 Notice that the picture above is in the connection with the statement.
\end{remark}

\end{proposition}

\begin{proof}
For $n\in\mathbb{N}$, define
\begin{align}
 I_n&\coloneqq 2+\tan(A_n+B_{n+1})+\tan(A_{n+1}+B_{n+1}),\label{def_I_n}\\
 P_n&\coloneqq \tan(A_{n+1}+B_{n+1})-\tan(A_n+B_{n+1},)\label{def_P_n}\\
 P_n'&\coloneqq \tan(A_n+B_{n+1})-\tan(A_n+B_n).\label{def_P_n'}
\end{align}
Furthermore, set
\begin{align}
\delta_n &\coloneqq \frac{I_n\cdot (1+\tan(A_n+B_n))\cdot (1+\tan(A_{n+1}+B_{n+1}))}{P_n^2\cdot(1+\tan(A_n+B_{n+1}))+P_n' I_n \cdot(1+\tan(A_{n+1}+B_{n+1}))},\label{def_del_n}\\
x_n&\coloneqq 2\delta_n\cdot\frac{1+\tan(A_n+B_{n+1})}{\tan(A_n+B_{n+1})}\cdot \frac{P_n}{I_n}+\frac{1}{\tan(A_n+B_{n+1})},\label{def_x_n}\\
y_n&\coloneqq \frac{1+\tan(A_n+B_n)-\delta_n\cdot P_n'}{1+\tan(A_n+B_{n+1})},\label{def_y_n}\\
c_n&\coloneqq (\delta_n + y_n)\cdot\tan(A_n+B_{n+1})+y_n,\label{def_c_n}\\
d_n&\coloneqq (\delta_n -y_n)\cdot \tan(A_{n+1}+B_{n+1})-y_n\label{def_d_n}.
\end{align}

For the convenience, let us denote the denominator of $\delta_n$ by $H_n$,
\begin{align}
 H_n\coloneqq P_n^2\cdot(1+\tan(A_n+B_{n+1}))+P_n' I_n \cdot(1+\tan(A_{n+1}+B_{n+1}))\label{def_H_n}.
\end{align}

First, from Lemma \ref{lema4} (Appendix \ref{Ap}), we have that $I_n$, $P_n$, $P_n'$, $\delta_n$, $x_n$, $y_n$, $c_n$ and $d_n$ are positive and $c_n<d_n$.

Let us prove $V_n\subset\Delta_n^{\delta_n}$. What is more, we show $V_n=\Delta_n^{\delta_n}$. Recall, we defined the set $\Delta_{n}^{\delta_n}$ in the way
\begin{align*}
\Delta_{n}^{\delta_n}=\{(y_1,y_2):\ &y_1=\delta_n,
(\delta_n+1)\cdot\tan(A_n+B_n)+1 <\\
& <y_2<(\delta_n+x_n)\cdot\tan(A_n+B_{n+1})-1\}.
\end{align*}

The definition \eqref{def_y_n} yields
\begin{align*}
 (\delta_n + y_n)\cdot\tan(A_n+B_{n+1})+y_n=(\delta_n+1)\cdot\tan(A_n+B_n)+1.
\end{align*}

Therefore, using \eqref{def_c_n} the equality
\begin{align*}
c_n=(\delta_n+1)\cdot\tan(A_n+B_n)+1
\end{align*}
holds.

By Lemma \ref{lema4} (Appendix \ref{Ap}) we have $d_n=(\delta_n+x_n)\cdot\tan(A_n+B_{n+1})-1$, hence $V_n=\Delta_{n}^{\delta_n}$.

Now, for $x=(x_1, x_2)\in U_n$, we prove $V_n\subset \tilde{S}_{n+1}^{\delta_n}(-x_1,x_2)$.
Because of geometrical reasoning, it is enough to show
\begin{align*}
V_n\subset \tilde{S}_{n+1}^{\delta_n}(-y_n,-y_n)\cap 
\tilde{S}_{n+1}^{\delta_n}(y_n,y_n).
\end{align*}

Let $p_1^n$ and $p_2^n$ be the lines given by
\begin{align*}
 &y-y_n=\tan(A_n+B_{n+1})\cdot(x+y_n),\\
 &y+y_n=\tan(A_{n+1}+B_{n+1})\cdot(x-y_n),
\end{align*}
respectively.
As in the picture, $p_1^n$ and $p_2^n$ are the lines through the points $(-y_n,y_n)$ and $(y_n,-y_n)$, respectively.

Denote by $P_1^{n}$ and $P_2^{n}$ the points which determine the intersection of lines $p_1^n$ and $p_2^n$ with the
set $\{(y_1,y_2): y_1=\delta_n, y_2>0\}$:
\begin{align*}
 P_1^{n}&=(x_{P_1}, y_{P_1})=(\delta_n, (\delta_n + y_n)\cdot \tan(A_n+B_{n+1})+y_n),\\
 P_2^{n}&=(x_{P_2}, y_{P_2})=(\delta_n, (\delta_n-y_n)\cdot\tan(A_{n+1}+B_{n+1})-y_n).
\end{align*}

Notice that $P_1^{n}$ and $P_2^{n}$ determine the lower boundary point and
the upper boundary point of the sets $\tilde{S}_{n+1}^{\delta_n}(y_n,y_n)$ and
$\tilde{S}_{n+1}^{\delta_n}(-y_n,-y_n)$, respectively.
Notice that showing $V_n= \{\delta_n\}\times (y_{P_1},y_{P_2})$ finishes the proof.
Clearly, the definitions \eqref{def_c_n} and \eqref{def_d_n} yield
\begin{align*}
 c_n&=y_{P_1},\\
 d_n&=y_{P_2},
\end{align*}
and hence the proposition.
\end{proof}

\includegraphics[scale=1]{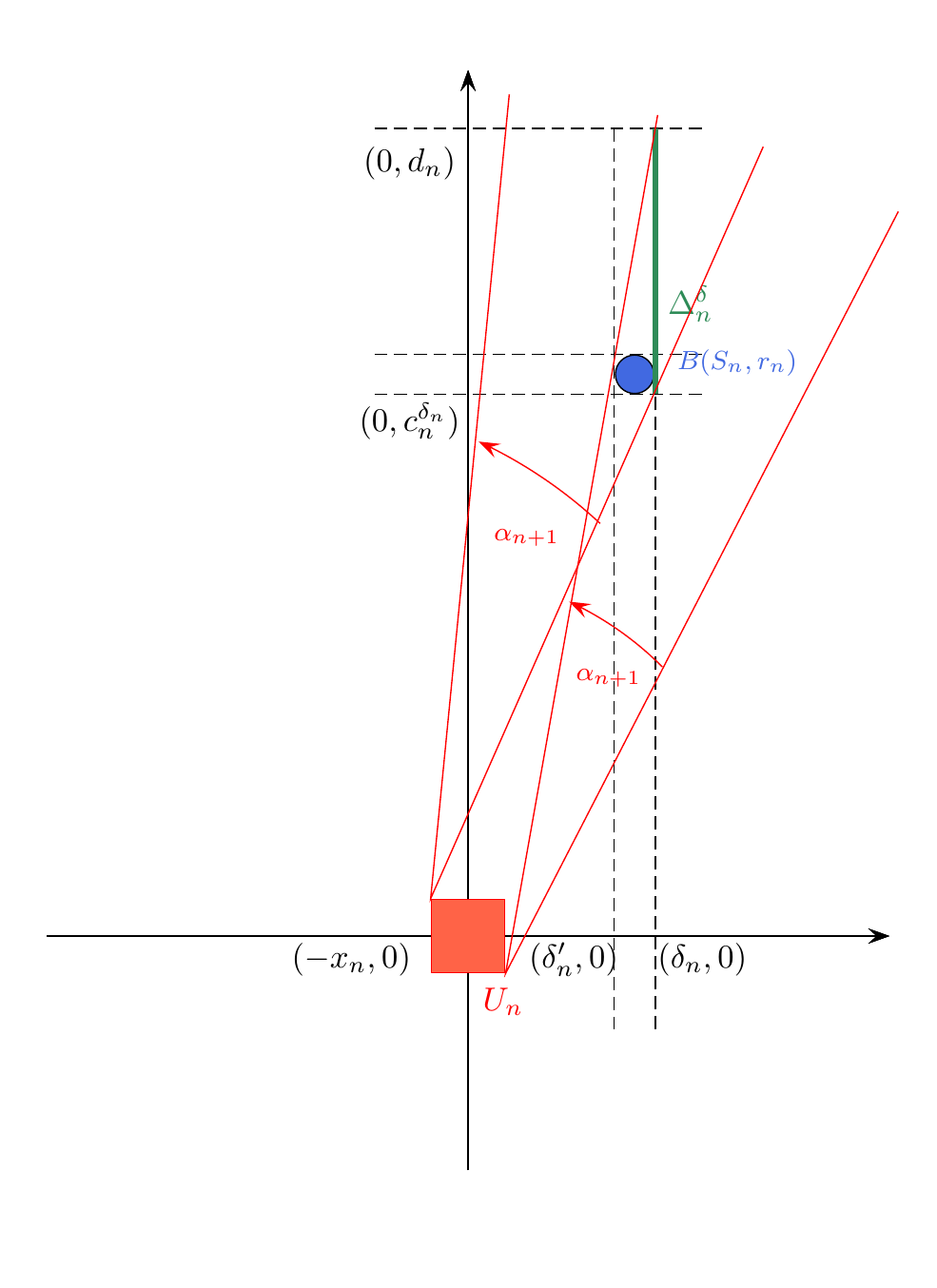}

\begin{proposition}{\label{prop5}}
Let $(\alpha_n)$, $(\beta_n)$ be sequences as in the Lemma \ref{lema} and let
$(I_n)$, $(P_n)$, $(P_n')$, $(\delta_n)$, $(x_n)$, $(y_n)$, $(c_n)$, $(d_n)$ and $(U_n)$ be as in the
Proposition \ref{prop4}.
There exist sequences $(\delta_n')$ and $(B(S_n, r_n))$, where:
\begin{enumerate}
\item $(\delta_n')$  is such that $\delta_n'<\delta_n$, for every $n\in\mathbb{N}$,
\item $B(S_n, r_n)$ is a ball with the center $S_n=(x_{S_n}, y_{S_n})$ and radius $r_n$, such that
$$
B(S_n, r_n)\subset \bigcup_{\{\varepsilon : \delta_n'<\varepsilon <\delta_n\}}
\Big(\tilde{S}_{n+1}^{\varepsilon}(-x_1,x_2)\cap \Delta_n^{\varepsilon}\Big),
$$
for every $x=(x_1, x_2)\in U_n$.
\end{enumerate}
\end{proposition}

\begin{remark}
Notions from the Proposition are illustrated at the picture above.
\end{remark}

\begin{proof}
For $n\in\mathbb{N}$, define
\begin{align}
\delta_n'&\coloneqq \frac{\delta_n\cdot(1+\tan(A_n+B_{n+1}))+y_n I_n}{1+\tan(A_{n+1}+B_{n+1})},\label{def_del_n_c}\\
J_n&\coloneqq 2\tan(A_n+B_{n+1})\cdot\tan(A_{n+1}+B_{n+1})\notag\\
    &\phantom{=\ } +\tan(A_n+B_{n+1})+\tan(A_{n+1}+B_{n+1}).\label{def_J_n}
\end{align}
Set
\begin{align}
&x_{S_n}\coloneqq \frac{\delta_n+y_n}{2}\cdot\frac{I_n}{1+\tan(A_{n+1}+B_{n+1})},\label{def_x_S}\\
&y_{S_n}\coloneqq \frac{\delta_n+y_n}{2}\cdot\frac{J_n}{1+\tan(A_{n+1}+B_{n+1})},\label{def_y_S}\\
&r_n\coloneqq \frac{1}{2}\cdot\frac{\delta_n  P_n-y_n 
I_n}{1+\tan(A_{n+1}+B_{n+1})}\label{def_r_n}.
\end{align}
By the first part of Lemma \ref{lema5} (Appendix \ref{Ap}),
\begin{align*}
B(S_n, r_n)\subset Q_n:=(\delta_n', \delta_n)\times (c_n,
c_n+(\delta_n-\delta_n')).
\end{align*}
Now, it is enough to show that for every $x=(x_1, x_2)\in U_n$
\begin{align*}
Q_n \subset \bigcup_{\{\varepsilon : \delta_n'<\varepsilon <\delta_n\}}\Big( \tilde{S}_{n+1}^{\varepsilon}(-x_1,x_2)\cap \Delta_n^{\varepsilon}\Big).
\end{align*}

Let $z=(z_1,z_2)\in Q_n$. We show
\begin{align*}
z \in \tilde{S}_{n+1}^{\varepsilon}(-x_1,x_2)\cap \Delta_n^{\varepsilon},
\end{align*}
for some $\delta_n'<\varepsilon<\delta_n$.

Recall,
\begin{align*}
\Delta_{n}^{\varepsilon}=\{(y_1,y_2):\ &y_1=\varepsilon,
(\varepsilon+1)\cdot\tan(A_n+B_n)+1<\\
&<y_2<(\varepsilon+x_n)\cdot\tan(A_n+B_{n+1})-1\},
\end{align*}
so in order to show $z\in \Delta_{n}^{\varepsilon}$, we prove that
\begin{align*}
c_n&\geq(\varepsilon+1)\cdot\tan(A_n+B_n)+1,\\
c_n+(\delta_n-\delta_n')&\leq(\varepsilon+x_n)\cdot\tan(A_n+B_{n+1})-1,
\end{align*}
for some $\delta_n'<\varepsilon<\delta_n$. Since the proof is quite technical,
we spell out the details in the second part of Lemma \ref{lema5} (Appendix \ref{Ap}).

Now, for $x=(x_1, x_2)\in U_n$, let us prove $z \in \tilde{S}_{n+1}^{\varepsilon}(-x_1,x_2)$.
Due to geometrical reasoning, it is enought to show that
\begin{align*}
z\in \tilde{S}_{n+1}^{\varepsilon}(-y_n,-y_n)\cap \tilde{S}_{n+1}^{\varepsilon}(y_n,y_n).
\end{align*}

As before, let $p_1^n$ and $p_2^n$ denote the lines
\begin{align*}
 &y-y_n=\tan(A_n+B_{n+1})\cdot(x+y_n),\\
 &y+y_n=\tan(A_{n+1}+B_{n+1})\cdot(x-y_n),
\end{align*}
respectively.
As before, the lines $p_1^n$ and $p_2^n$ are the lines through the points $(-y_n,y_n)$ and $(y_n,-y_n)$, respectively.

Denote by $P_1^{n}$ and $P_2^{n}$ the points which determine the intersection of $p_1^n$ and $p_2^n$ with the set $\{(y_1,y_2): y_1=\varepsilon, y_2>0\}$:
\begin{align}
 P_1^{n}&=(x_{P_1}, y_{P_1})=(\varepsilon, (\varepsilon + y_n)\cdot \tan(A_n+B_{n+1})+y_n),\label{y_P_1}\\
 P_2^{n}&=(x_{P_2}, y_{P_2})=(\varepsilon,(\varepsilon-y_n)\cdot\tan(A_{n+1}+B_{n+1})-y_n)\label{y_P_2}.
\end{align}

Notice that $P_1^{n}$ and $P_2^{n}$ determine the lower boundary point and
the upper boundary point of sets $\tilde{S}_{n+1}^{\varepsilon}(y_n,y_n)$ and
$\tilde{S}_{n+1}^{\varepsilon}(-y_n,-y_n)$, respectively.
If we show $z_2\in  (y_{P_1},y_{P_2})$, the proposition is proved.

To this end, we prove
\begin{align*}
c_n &\geq y_{P_1},\\
c_n + (\delta_n-\delta_n')&\leq y_{P_2},
\end{align*}
and the details can be found in the Lemma \ref{har1} in Appendix \ref{Ap}.
\end{proof}

\subsection{Harnack Inequality}


In the subsection that follows, we give a proof of the fact that Harnack inequality does not hold. 

\includegraphics[scale=0.8]{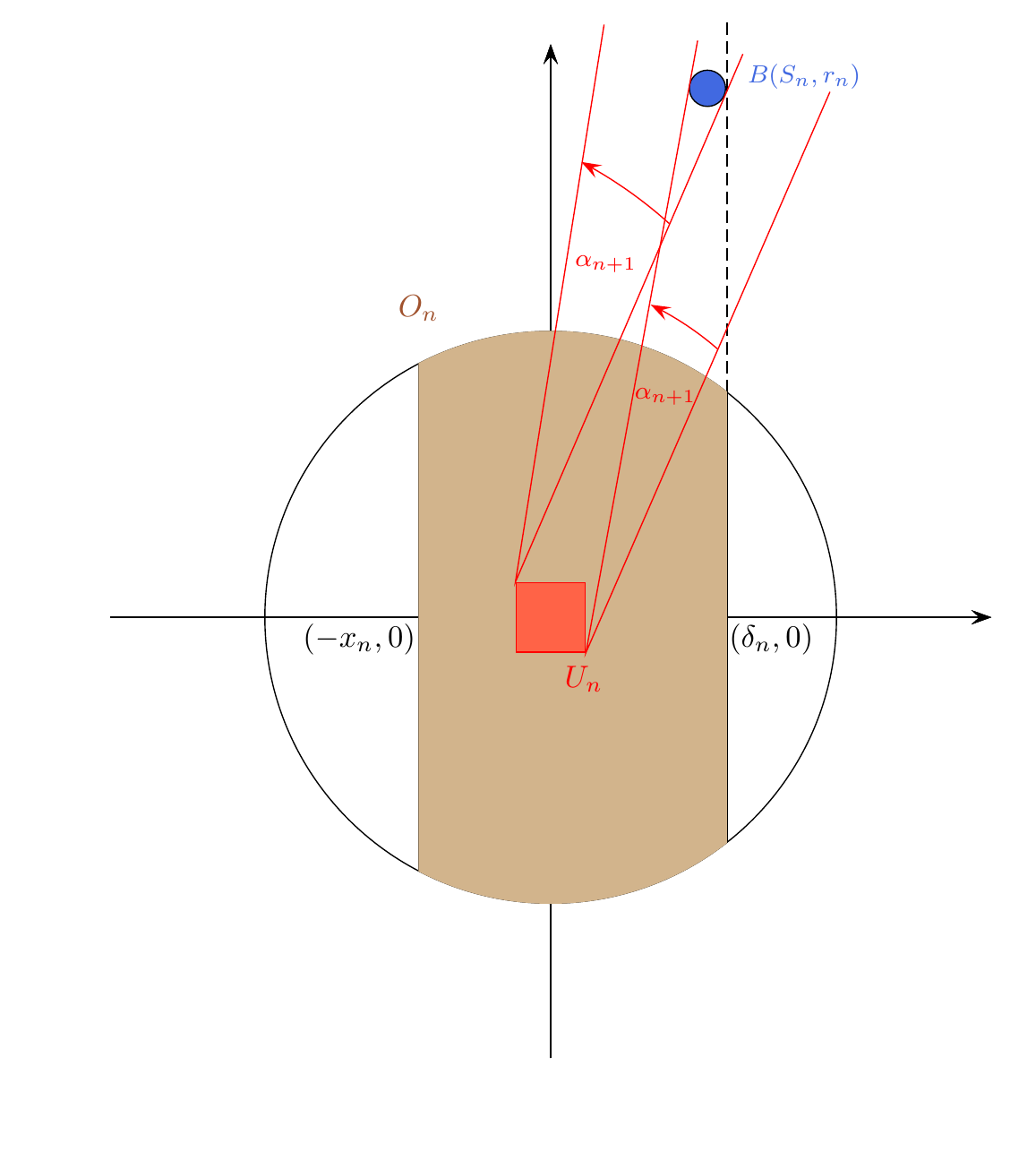}

Let $X$ be a symmetric $\alpha$-stable L\'evy process in $\mathbb{R}^2$, $0<\alpha<1$,
with the characteristic function of the form
\begin{align*}
\mathbb{E}^0 e^ {iu\cdot X_t} =  e^{-t \Phi(u)}, \ u \in\mathbb{R}^2,
\end{align*}
where
\begin{align*}
\Phi(u)=\int_{\mathbb{S}^{1}} |u\cdot \xi|^\alpha \mu(\mathrm{d}\xi).
\end{align*}
We define the spectral measure $\mu$  appropriately, below in Theorem \ref{har_thm_mat}.
Furthermore, for a given starting point $x\in B_1$ of the process $X$, we examine the probability
\begin{align*}
\mathbb{P}^{x}(X_{\tau_{B_1}} \in B_n),\ n\in \mathbb{N}
\end{align*}
of hitting the set $B_n=B(S_n,r_n)$ at the first exit time of a unit ball $B_1$, where for every $n\in \mathbb{N}$, $B_n$ are as in Proposition \ref{prop5}. We choose the points $w_0=(1/2,0)$, $0=(0,0)$ and
define the harmonic functions
\begin{align*}
 u_n(x)= \mathbb{P}^{x}(X_{\tau_{B_1}} \in B_n),\ n\in\mathbb{N},\ x\in B_1.
\end{align*}
We show
\begin{align*}
 \frac{u_n(w_0)}{u_n(0)}\leq c(\alpha)\cdot
a_n b_n ,
\end{align*}
for sequences $(a_n)$ and $(b_n)$ with certain properties and eventually obtain the desired conclusion.

\begin{remark}
Constants are positive real numbers, which exact value may
vary from one line to the other. For the convenience they will not be explicitly
stated throughout the proof.
\end{remark}

Before we prove the Theorem \ref{har_thm_mat}, we state the proposition for the Green function estimate of the unit ball $B_1$ for $X$ (see \cite{bg-sz-1}, Theorem 2).

\begin{proposition}\label{Gr_f_est}
Let $X$ be a symmetric $\alpha$-stable L\'evy process in $\mathbb{R}^d$, $0<\alpha<2$, with the characteristic function of the form
\begin{align*}
\mathbb{E}^0 e^ {iu\cdot X_t} =  e^{-t \Phi(u)}, \ u \in\mathbb{R}^d,
\end{align*}
where
\begin{align*}
\Phi(u)=\int_{\mathbb{S}^{d-1}} |u\cdot \xi|^\alpha \mu(\mathrm{d}\xi),
\end{align*}
and $\mu$ is a finite, symmetric measure on $\mathbb{S}^{d-1}$ such that $\mu(S^{d-1})>0$. Let $\mu$ be absolutely continuous with respect to the surface measure on $\mathbb{S}^{d-1}$, and denote by $f_\mu$ its density, for which the inequality
\begin{align*}
 0\leq f_\mu(\xi)\leq m,\ \xi\in\mathbb{S}^{d-1},
\end{align*}
holds, for some $m>0$.
Then the estimate for the Green function $G_{B_1}$
\begin{align*}
c^{-1}\cdot s(y)\cdot|x-y|^{\alpha-d} \leq G_{B_1}(x,y)\leq c\cdot s(y)\cdot|x-y|^{\alpha-d},
\end{align*}
holds, for all $|x|<1/2, |y|<1$.
Here, $s(y)=\mathbb{E}^y [\tau_{B_1}]$ is the expected time spent in $B_1$, if the process starts from $y\in B_1$.
\end{proposition}

\begin{theorem}\label{har_thm_mat}
Let $X$ be a symmetric $\alpha$-stable L\'evy process in $\mathbb{R}^2$,\\
$0<\alpha<1$,
with the characteristic function of the form
\begin{align*}
\mathbb{E}^0 e^ {iu\cdot X_t} =  e^{-t \Phi(u)}, \ u \in\mathbb{R}^2,
\end{align*}
where
\begin{align*}
\Phi(u)=\int_{\mathbb{S}^{1}} |u\cdot \xi|^\alpha \mu(\mathrm{d}\xi).
\end{align*}
Let the 
measure $\mu$ on $\mathbb{S}^{1}$ be absolutely continuous with respect to the surface measure on $\mathbb{S}^{1}$, and denote by $f_\mu$ its density.
Set
\begin{align*}
 f_\mu&=\mathbf{1}_{\bigcup_{n\geq 1} (B_{\xi_n,r_n} \cup \ B_{-\xi_n,r_n})},
\end{align*}
where
\begin{align*}
  &B_{\xi_n,r_n}=B(\xi_n,r_n)\cap \mathbb{S}^1,\\
  &r_n=2\cdot\sin(\alpha_n/4),\\
  &\xi_n=\big(\cos(A_{n-1}+B_n+\alpha_n/2\big), \sin\big(A_{n-1}+B_n+\alpha_n/2)\big) \in \mathbb{S}^1, n\in \mathbb{N},
\end{align*}
where $(\alpha_n)$ and $(\beta_n)$ are as in Lemma \ref{lema}.
Then the Harnack inequality for $X$ does not hold.
\end{theorem}
\begin{proof}
Let
\begin{align*}
u_n(x)&= \mathbb{P}^{x}(X_{\tau_B} \in B_n),\ n\in \mathbb{N},\ x\in B(0,1),\\
O_n&=\big((-x_n,\delta_n)\times(-1,1)\big)\cap B_1,
\end{align*}
where $B_n=B(S_n,r_n)$, $x_n$, $\delta_n$ are introduced in the Propositions \ref{prop4} and \ref{prop5}, and set $w_0=(1/2,0)$ and $0=(0,0)$.

By the strong Markov property, every function $u_n$ is harmonic in $B_1$ with respect to $X$.

Namely, for a set $U\subset \bar{U} \subset B_1$,
\begin{align*}
\mathbb{E}^x[u_n(X_{\tau_U})]&=\mathbb{E}^x[\mathbb{P}^{X_{\tau_U}}[X_{\tau_B}\in B_n]]\\
&= \mathbb{P}^x[X_{\tau_B}\circ \theta_{\tau_U}\in B_n]\\
&=\mathbb{P}^x[X_{\tau_B}\in B_n]\\
&=u_n(x).
\end{align*}

Using the L\'evy system formula (\cite{chen})
it follows
\begin{align*}
&\mathbb{E}^{w_0}\Bigg[\sum_{s\leq t\wedge \tau_B}\mathbf{1}_{\{X_{s-}\in O_n, X_{s} \in B_n\}}\Bigg]\\
&=\mathbb{E}^{w_0}\Bigg[\int_{0}^{t\wedge \tau_B}\mathbf{1}_{O_n}(X_s)\Big(\int_{B_n} f_{\nu}(y-X_s)\ \mathrm{d}y\Big)\ \mathrm{d}s\Bigg] \\
&\leq\mathbb{E}^{w_0}\Bigg[\int_{0}^{t\wedge \tau_B}\mathbf{1}_{O_n}(X_s)\Big(\int_{B_n} |y-X_s|^{-\alpha-2}\ \mathrm{d}y\Big)\ \mathrm{d}s\Bigg]\\
&\leq |B_n|\cdot[(\delta_n+1)\cdot\tan(A_n+B_n)]^{-\alpha-2}\cdot\mathbb{E}^{w_0}\Bigg[\int_0^{t\wedge\tau_B} \mathbf{1}_{O_n}(X_s)\ \mathrm{d}s\Bigg]\\
&\leq |B_n|\cdot [(\delta_n+1)\cdot\tan(A_n+B_n)]^{-\alpha-2}\cdot \mathbb{E}^{w_0}\Bigg[\int_0^{\infty} \mathbf{1}_{O_n}(X_s)\ \mathrm{d}s\Bigg].
\end{align*}

The estimates of the potential kernel (\cite{sz-1}) imply
\begin{align*}
\mathbb{E}^{w_0}\Bigg[\int_0^{\infty} \mathbf{1}_{O_n}(X_s)\ \mathrm{d}s\Bigg]&\leq c(\alpha)\cdot \int_ {O_n} |y-w_0|^{\alpha-2}\ \mathrm{d}y\\
&\leq c(\alpha)\cdot(\delta_n+ x_n).
\end{align*}
From here we have
\begin{align*}
&\mathbb{E}^{w_0}\Bigg[\sum_{s\leq t\wedge \tau_B}\mathbf{1}_{\{X_{s-}\in O_n, X_{s} \in B_n\}}\Bigg]\\
&\leq c(\alpha)\cdot|B_n|\cdot [(\delta_n+1)\cdot\tan(A_n+B_n)]^{-\alpha-2}\cdot (\delta_n+x_n).
\end{align*}

Letting $t\rightarrow\infty$, by the dominated convergence theorem, we obtain
\begin{align}\label{gornja_og_alt}
u_n(w_0)
\leq
c_1(\alpha,n)\cdot(\delta_n+ x_n),
\end{align}
where
\begin{align*}
c_1(\alpha,n)=c(\alpha)\cdot|B_n|\cdot [(\delta_n+1)\cdot\tan(A_n+B_n)]^{-\alpha-2}.
\end{align*}

Let us compute the lower bound on $u_n(0)$.

By the L\'evy system formula (see \cite{chen})
and the construction, it follows that
\begin{align*}
u_n(0)&=\mathbb{P}^{0}(X_{\tau_B} \in B_n) \\
&\geq \mathbb{E}^{0}\Bigg[\sum_{s\leq t\wedge \tau_B}\mathbf{1}_{\{X_{s-}\in U_n, X_{s} \in B_n\}}\Bigg]\\
&=\mathbb{E}^{0}\Bigg[\int_{0}^{t\wedge\tau_B}\mathbf{1}_{U_n}(X_s)\Big(\int_{B_n} f_{\nu}(y-X_s)\ \mathrm{d}y\Big)\ \mathrm{d}s\Bigg]\\
&=\mathbb{E}^{0}\Bigg[\int_{0}^{t\wedge\tau_B}\mathbf{1}_{U_n}(X_s)\Big(\int_{B_n} |y-X_s|^{-\alpha-2}\ \mathrm{d}y\Big)\ \mathrm{d}s\Bigg]\\
&\geq \bar{c}_2(\alpha,n)\cdot\mathbb{E}^{0}\Bigg[\int_{0}^{t\wedge\tau_B}\mathbf{1}_{U_n}(X_s)\ \mathrm{d}s\Bigg],
\end{align*}
where
\begin{align*}
 \bar{c}_2(\alpha,n)=|B_n|\cdot[2\delta_n-\delta_n'+2y_n+ 1 +(\delta_n+1)\cdot\tan(A_n+B_n)]^{-\alpha-2}.
\end{align*}

Letting $t\rightarrow\infty$, by the dominated convergence theorem, it follows
\begin{align*}
u_n(0)\geq \bar{c}_2(\alpha,n)\cdot\mathbb{E}^{0}\Bigg[\int_{0}^{\tau_B}\mathbf{1}_{U_n}(X_s)\ \mathrm{d}s\Bigg].
\end{align*}

Using the the estimate of the Green function from \cite{bg-sz-1}
there is $c(\alpha)$ so that
\begin{align*}
\mathbb{E}^{0}\Bigg[\int_{0}^{\tau_B}\mathbf{1}_{U_n}(X_s)\ \mathrm{d}s\Bigg]&=\mathbb{E}^{0}\Bigg[\int_{0}^{\infty}\mathbf{1}_{\{X_s\in U_n,\tau_B>s\}}\ \mathrm{d}s\Bigg]\\
&=\int_{0}^{\infty}\mathbb{P}^0(X_s\in U_n,\tau_B>s)\ \mathrm{d}s\\
&=\int_{U_n}G_B(0,y)\ \mathrm{d}y \\
&\geq \int_{U_n}s(y)|y|^{\alpha-2}\ \mathrm{d}y\\
&\geq c(\alpha)\cdot (1-2y_n)^\alpha \int_{U_n}|y|^{\alpha-2}\ \mathrm{d}y \\
&\geq c(\alpha)\cdot (1-2y_n)^\alpha \int_{B(0,y_n/2)}|y|^{\alpha-2}\ \mathrm{d}y \\
&\geq c(\alpha)\cdot \big(1-2 y_n\big)^{\alpha}\big(y_n\big)^{\alpha}.
\end{align*}

From here we obtain the lower bound
\begin{align}\label{donja_alt}
u_n(0)\geq c_2(\alpha,n)\cdot \big(1-2 y_n\big)^{\alpha}\big(y_n\big)^{\alpha},
\end{align}
where
\begin{align*}
&c_2(\alpha,n)=c(\alpha)|B_n|\cdot
[2\delta_n-\delta_n'+2y_n+ 1 +(\delta_n+1)\cdot\tan(A_n+B_n)]^{-\alpha-2}.
\end{align*}

Combining inequalities \eqref{gornja_og_alt} and \eqref{donja_alt}, we obtain:

\begin{align}\label{nejedn_alt}
\frac{u_n(w_0)}{u_n(0)}\leq c(\alpha)\cdot a_n b_n,
\end{align}
where
\begin{align*}
a_n&=\Bigg(\frac{2\delta_n-\delta_n'+2y_n+1}{(\delta_n+1)\cdot\tan(A_n+B_n)}+1\Bigg)^{\alpha+2},\\
b_n&=\frac{\delta_n+x_n}{y_n^\alpha\cdot (1-2 y_n)^\alpha}.
\end{align*}
Due to the construction,
\begin{align}\label{limes_1_alt}
\lim_{n\rightarrow \infty} a_n=1.
\end{align}

Using $\eqref{uvjet}$ and $\eqref{uvjet_d}$, we obtain
\begin{align}\label{limsup_alt}
\limsup_{n\rightarrow \infty} \frac{\delta_n+x_n}{y_n}<\infty.
\end{align}
Since the proof of \eqref{limsup_alt} is similar to the proof of Lemma \ref{har} (Appendix), we skip it.

Since $0<\alpha<1$,
\begin{align*}
b_n&=\frac{\delta_n+x_n}{y_n^\alpha\cdot (1-2 y_n)^\alpha}\\
&=\frac{1}{(1-2 y_n)^\alpha}\cdot \frac{\delta_n+x_n}{y_n}\cdot y_n^{1-\alpha}.
\end{align*}

Therefore,
\begin{align}\label{limes_2_alt}
\lim_{n\rightarrow \infty} b_n=0.
\end{align}

To conclude, $\eqref{nejedn_alt}$, $\eqref{limes_1_alt}$ and $\eqref{limes_2_alt}$ imply
\begin{align*}
\frac{u_n(0)}{u_n(w_0)}
\end{align*}
can be made as large as we like by taking $n$ large enough. This implies that
the Harnack inequality for $X$ is not possible.
\end{proof}

\begin{remark}
 Notice that for $1<\alpha<2$, we have
 \begin{align*}
  \frac{u_n(0)}{u_n(w_0)}\geq c(\alpha) \cdot \frac{1}{a_n} \cdot \frac{1}{b_n},
 \end{align*}
where
\begin{align*}
 \frac{1}{b_n}=(1-2y_n)^\alpha \cdot \frac{y_n}{\delta_n+x_n} \cdot y_n^{\alpha-1}.
\end{align*}
Notice that from here, we have
\begin{align*}
\lim_{n\rightarrow \infty} \frac{1}{b_n}=0,
\end{align*}
therefore the proof breaks down.
\end{remark} 

\begin{remark}
In connection with the article \cite{bg-sz-1}, there it was shown that as in our case ($d=2$), for $1<\alpha<2$ Harnack inequality holds (see Corollary 13 of the aforementioned article).
\end{remark}


\hfill \break

\section{Weak Harnack Inequality}

\begin{definition}
\emph{The weak Harnack inequality} for a symmetric $\alpha$-stable L\'evy processes $X$ holds if there is a constant $C\geq 1$ such that for every non-negative function $u\colon \mathbb{R}^d\rightarrow \mathbb{R}$, which
is harmonic in $B_1$ with respect to $X$,
the inequality
\begin{align*}
 \|u\|_{L^{1}(B_{1/2})} \leq C\cdot \inf_{B_{1/4}} u
\end{align*}
holds.
\end{definition}

We continue with the definitions regarding the spherical part of the L\'{e}vy measure (see \cite{bg-sz-2} (cf. \cite{priola})).
\begin{definition}
A measure $\lambda$ on $\mathbb{R}^d$
is called \emph{degenerate} if there is a proper linear subspace $M$ of $\mathbb{R}^d$ such that $Spt(\lambda)\subset M$, where $Spt(\lambda)$
denotes the support of the measure $\mu$.

A measure $\lambda$ is called \emph{non-degenerate} if it is not degenerate. 
\end{definition}

\begin{definition}\label{def_ned}
 A measure $\mu$ on $\mathbb{S}^{d-1}$ is called a \emph{spectral measure} if it is positive, finite, non-degenerate and symmetric.
\end{definition}

For the equivalence of the non-degeneracy of the measure $\mu$ and the condition
\begin{align}\label{nondeg}
\Phi(u)\geq c\cdot |u|^\alpha, \ u\in \mathbb{R}^d,
\end{align}
where $c=c(\alpha)$ is a positive constant and $\Phi$ is as in \eqref{k_e}, see \cite{priola}.

\begin{theorem}[Weak Harnack Inequality]{\label{Weak Harnack}}
Let $X$ be a symmetric $\alpha$-stable L\'evy process in $\mathbb{R}^d$, $d\geq 2$, with index of stability $\alpha \in (0,2)$ and the characteristic function of the form
\begin{align*}
\mathbb{E}^0 e^ {iu\cdot X_t} =  e^{-t \Phi(u)}, \ \ u \in
\mathbb{R}^d,\ t\geq 0,
\end{align*}
where the characteristic exponent is given by
\begin{align*}
\Phi(u)=\int_{\mathbb{S}^{d-1}} |u\cdot \xi|^\alpha \mu(\mathrm{d}\xi),
\end{align*}
and $\mu$ is a spectral measure.
Furthermore, let $\mu$ be absolutely continuous with respect to the uniform measure $\sigma$ on the sphere $\mathbb{S}^{d-1}$ and denote by $f_\mu$ its density.
Assume
that there is a positive constant $m$ such that
\begin{align*}
0\leq f_{\mu}(\xi)\leq m,\  \xi \in \mathbb{S}^{d-1}.
\end{align*}
Then the weak Harnack inequality for $X$ holds.
\end{theorem}

\begin{lemma}\label{prva_gornja}
Let $\alpha\in (0,2), d\geq 2$. There is a constant $c_1=c_1(\alpha,d)$ such that for every $|w|<3/4$, the inequality
  \begin{align*}
   \int_{B_{1/2}} G_{B_{3/4}}(x,w)\ \mathrm{d}x\leq c_1
  \end{align*}
  holds.
\end{lemma}
\begin{proof}
By the inequality $p_D\leq p$ and the estimate of the transition density $p$ for small times (see e.g. \cite{sz-1}, Theorem 1), it follows:
\begin{align*}
&\int_{B_{1/2}} G_{B_{3/4}}(x,w)\ \mathrm{d}x\\
&= \int_{B_{1/2}} \bigg[\int_0 ^\infty p_{B_{3/4}}(t,x,w)\ \mathrm{d}t\bigg]\ \mathrm{d}x \\
&\leq\int_{B_{1/2}} \bigg[\int_0 ^\infty p(t,x,w)\ \mathrm{d}t\bigg]\ \mathrm{d}x\\
&\leq c(\alpha,d) \int_{B_{1/2}} \bigg[\int_0 ^1 \bigg(t^{-d/\alpha} \wedge \frac{t}{|x-w|^{\alpha+d}}\bigg)\ \mathrm{d}t\bigg]\ \mathrm{d}x\\
&\text{\;\;\;\;\;}+\int_{B_{1/2}} \bigg[\int_1 ^\infty t^{-d/\alpha}\cdot p(1, t^{-1/\alpha}x, t^{-1/\alpha}w)\ \mathrm{d}t\bigg]\ \mathrm{d}x\\
&\leq c(\alpha,d) \int_{B_{1/2}} \bigg[\int_0 ^{|x-w|^\alpha\wedge 1} \frac{t}{|x-w|^{\alpha+d}}\ \mathrm{d}t\bigg]\ \mathrm{d}x\\
&\text{\;\;\;\;\;} + c(\alpha,d) \int_{B_{1/2}} \bigg[\int_{|x-w|^\alpha\wedge 1}^1 t^{-d/\alpha}\ \mathrm{d}t\bigg]\ \mathrm{d}x\\
&\text{\;\;\;\;\;}+c(\alpha,d) \bigg[\int_{B_{1/2}} \int_1 ^\infty t^{-d/\alpha}\ \mathrm{d}t\bigg]\ \mathrm{d}x\\
&\leq c(\alpha,d) \Bigg[ \int_{B_{1/2}} |x-w|^{\alpha-d}\ \mathrm{d}x\\
&\text{\;\;\;\;\;\;\;\;\;\;\;\;\;\;\;\;\;}+\frac{\alpha}{d-\alpha}\int_{B_{1/2}}\bigg[\big(|x-w|^\alpha\wedge 1\big)^{\frac{\alpha-d}{\alpha}}-1\bigg]\ \mathrm{d}x\Bigg]\\
&\text{\;\;\;\;\;}+c(\alpha,d)\ \frac{\alpha}{d-\alpha}|B_{1/2}|\\
&\leq c(\alpha,d).
\end{align*}
\end{proof}

\begin{lemma}\label{prva_donja_1}
Let $\alpha \in (0,2), d\geq 2$. There exist $\delta_1=\delta_1(\alpha,d)>0$ 
and $c_2=c_2(\alpha,d,\delta_1)$ such that for every $|\bar{x}|<1/4$ and every $w\in B(\bar{x},\delta_1)$
the inequality
\begin{align*}
G_{B_{3/4}}(\bar{x},w)\geq c_2
\end{align*}
holds.
\end{lemma}

\begin{proof}

Using
\begin{align}
G_{B_{3/4}}(\bar{x},w)=G(\bar{x},w)-\mathbb{E}^{\bar{x}}[G(X_{\tau_{B_{3/4}}},w)],\label{formula}
\end{align}
in order to prove the lemma, we compute the estimates for $G(\bar{x},w)$ from below and $\mathbb{E}^{\bar{x}}[G(X_{\tau_{B_{3/4}}},w)]$ from above.
Using the heat kernel estimates for small times (\cite{sz-1}, Theorem 1), we obtain
\begin{align}
&\mathbb{E}^{\bar{x}}[G(X_{\tau_{B_{3/4}}},w)]\notag\\
&=\int_{{B_{3/4}}^c}\ G(u,w)P_{B_{3/4}}(\bar{x},u)\ \mathrm{d}u\notag\\
&= \int_{{B_{3/4}}^c}\ P_{B_{3/4}}(\bar{x},u)\bigg[\int_{0}^{\infty} p(t,u,w)\ \mathrm{d}t\bigg]\ \mathrm{d}u\notag\\
&= \int_{{B_{3/4}}^c}\ P_{B_{3/4}}(\bar{x},u)\bigg[\int_{0}^{1} p(t,u,w)\ \mathrm{d}t\bigg]\ \mathrm{d}u\notag\\
&\ + \int_{{B_{3/4}}^c}\ P_{B_{3/4}}(\bar{x},u)\bigg[\int_{1}^{\infty} p(t,u,w)\ \mathrm{d}t\bigg]\ \mathrm{d}u\notag\\
&\leq c(\alpha,d)\int_{{B_{3/4}}^c}\ P_{B_{3/4}}(\bar{x},u) \bigg[\int_{0}^{1} \bigg(t^{-d/\alpha} \wedge \frac{t}{|u-w|^{\alpha+d}}\bigg)\ \mathrm{d}t\bigg]\ \mathrm{d}u\notag\\
&+ \int_{{B_{3/4}}^c}\ P_{B_{3/4}}(\bar{x},u)\bigg[\int_{1}^{\infty} t^{-d/\alpha}\cdot p(1, t^{-1/\alpha}u, t^{-1/\alpha}w)\ \mathrm{d}t\bigg]\ \mathrm{d}u\notag\\
&\leq c(\alpha,d)\int_{{B_{3/4}}^c}\ P_{B_{3/4}}(\bar{x},u)\bigg[ \int_{0}^{1} \bigg(t^{-d/\alpha} \wedge \frac{t}{|u-w|^{\alpha+d}}\bigg)\ \mathrm{d}t\bigg]\ \mathrm{d}u\notag\\
&+ c(\alpha, d)\int_{{B_{3/4}}^c}\ P_{B_{3/4}}(\bar{x},u)\bigg[\int_{1}^{\infty} t^{-d/\alpha}\ \mathrm{d}t\bigg]\ \mathrm{d}u\notag\\
&\leq c(\alpha,d)\int_{{B_{3/4}}^c}\ P_{B_{3/4}}(\bar{x},u)\bigg[\int_{0}^{1} \bigg(t^{-d/\alpha} \wedge \frac{t}{|u-w|^{\alpha+d}}\bigg)\ \mathrm{d}t\bigg]\ \mathrm{d}u\notag\\
&+ c(\alpha, d)\label{int}.
\end{align}

Examining the integral in \eqref{int} more closely, we obtain: 
\begin{align}
&\int_{{B_{3/4}}^c}\ P_{B_{3/4}}(\bar{x},u)\bigg[\int_{0}^{1} \bigg(t^{-d/\alpha} \wedge \frac{t}{|u-w|^{\alpha+d}}\bigg)\ \mathrm{d}t\bigg]\ \mathrm{d}u\notag\\
&=\int_{{B_{3/4}}^c}\ P_{B_{3/4}}(\bar{x},u)\mathbf{1}_{\{|w-u|<1\}}(u)\bigg[\int_{0}^{1} \bigg(t^{-d/\alpha} \wedge \frac{t}{|u-w|^{\alpha+d}}\bigg)\ \mathrm{d}t\bigg]\ \mathrm{d}u\notag\\
&+\int_{{B_{3/4}}^c}\ P_{B_{3/4}}(\bar{x},u)\mathbf{1}_{\{|w-u|>1\}}(u)\bigg[\int_{0}^{1} \bigg(t^{-d/\alpha} \wedge \frac{t}{|u-w|^{\alpha+d}}\bigg)\ \mathrm{d}t\bigg]\ \mathrm{d}u\notag\\
&=I_1+I_2.\label{skupa}
\end{align}

\begin{align}
 I_1&=\int_{{B_{3/4}}^c}\ P_{B_{3/4}}(\bar{x},u)\mathbf{1}_{\{|w-u|<1\}}(u)\notag\\
 &\text{\;\;\;\;\;\;\;\;\;\;\;\;\;\;\;\;\;\;\;\;\;\;\;\;\;\;\;\;\;\;\;\;\;\;\;\;\;\;\;\;\;\;\;\;}\bigg[\int_{0}^{1} \bigg(t^{-d/\alpha} \wedge \frac{t}{|u-w|^{\alpha+d}}\bigg)\ \mathrm{d}t\bigg]\ \mathrm{d}u\notag\\
 &=\int_{{B_{3/4}}^c}\ P_{B_{3/4}}(\bar{x},u)\mathbf{1}_{\{|w-u|<1\}}(u)|w-u|^{-\alpha-d}\bigg[\int_{0}^{|w-u|^\alpha}t\ \mathrm{d}t\bigg]\ \mathrm{d}u\notag\\
 &+\int_{{B_{3/4}}^c}\ P_{B_{3/4}}(\bar{x},u)\mathbf{1}_{\{|w-u|<1\}}(u)\bigg[\int_{|w-u|^\alpha}^{1}t^{-d/\alpha}\ \mathrm{d}t\bigg]\ \mathrm{d}u\notag\\
 &=c_1 \int_{{B_{3/4}}^c}\ P_{B_{3/4}}(\bar{x},u)\mathbf{1}_{\{|w-u|<1\}}(u)|w-u|^{\alpha-d}\ \mathrm{d}u\notag\\
 &+c_2\int_{{B_{3/4}}^c}\ P_{B_{3/4}}(\bar{x},u)\mathbf{1}_{\{|w-u|<1\}}(u)\frac{\alpha}{d-\alpha}\big(|w-u|^{\alpha-d}-1\big)\ \mathrm{d}u\notag\\
 &\leq c_1(\alpha,d,\bar{\delta}_1),\label{prvi_int}
\end{align}
where in the last inequality we have used $w \in B(\bar{x},\bar{\delta}_1)$, for $\bar{\delta}_1>0$ small enough.

\begin{align}
 I_2&=\int_{{B_{3/4}}^c}\ P_{B_{3/4}}(\bar{x},u)\mathbf{1}_{\{|w-u|>1\}}(u)\notag\\
   &\text{\;\;\;\;\;\;\;\;\;\;\;\;\;\;\;\;\;\;\;\;\;\;\;\;\;\;\;\;\;\;\;\;\;\;\;\;\;}\bigg[\int_{0}^{1} \bigg(t^{-d/\alpha} \wedge \frac{t}{|u-w|^{\alpha+d}}\bigg)\ \mathrm{d}t\bigg]\ \mathrm{d}u\notag\\
   &= \int_{{B_{3/4}}^c}\ P_{B_{3/4}}(\bar{x},u)\mathbf{1}_{\{|w-u|>1\}}(u)|w-u|^{-\alpha-d}\bigg[\int_{0}^{1} t\ \mathrm{d}t\bigg]\ \mathrm{d}u\notag\\
   &= c\int_{{B_{3/4}}^c}\ P_{B_{3/4}}(\bar{x},u)\mathbf{1}_{\{|w-u|>1\}}(u)|w-u|^{-\alpha-d}\ \mathrm{d}u\notag\\
   &\leq c_2.\label{drugi_int}
\end{align}

In conclusion, by \eqref{int}, \eqref{skupa}, \eqref{prvi_int} and \eqref{drugi_int}, we obtain
\begin{align}\label{gornja}
 \mathbb{E}^{\bar{x}}[G(X_{\tau_{B_{3/4}}},w)]\leq \bar{c}_1(\alpha, d, \bar{\delta}_1),
\end{align}
for all $w \in B(\bar{x},\bar{\delta}_1)$ and $\bar{\delta}_1>0$ small enough.

To estimate $G(\bar{x},w)$ from below, we use the continuity of the potential density (see \cite{taylor}).

Due to
\begin{align*}
 p(1,0)\geq c>0,
\end{align*}
by continuity of $p(1, \cdot)$ in $x=0$, there is $R>0$ such that $p(1,x)>\frac{1}{2}\cdot p(1,0)$, for all $|x|<R$.

Furthermore, for $|\xi|=1$, since:
\begin{align*}
 G(0,\xi)&\geq\int_{R^{-\alpha}}^{\infty} p(t,\xi)\ \mathrm{d}t\\
       &= \int_{R^{-\alpha}}^{\infty}t^{-d/\alpha}\cdot p\Big(1,\frac{\xi}{t^{1/\alpha}}\Big)\ \mathrm{d}t > \frac{1}{2}\cdot \int_{R^{-\alpha}}^{\infty}t^{-d/\alpha}\cdot p(1,0) \ \mathrm{d}t\\
	&=c_1(\alpha,d)>0,
\end{align*}
we obtain
\begin{align}\label{ocjena}
 G(0,\xi)\geq c_1(\alpha,d).
\end{align}

For $|x|\neq 0$, by scaling and \eqref{ocjena}
\begin{align*}
 G(0,x)=|x|^{\alpha-d}\cdot G\bigg(0,\frac{x}{|x|}\bigg)\geq c_1 \cdot |x|^{\alpha-d}.
\end{align*}

Therefore,
\begin{align}\label{donja}
 G(\bar{x},w) \geq \bar{c}_2 \cdot |\bar{x}-w|^{\alpha-d}.
\end{align}

Now, choose $\delta_1$ such that $\delta_1 < (\bar{c}_2/(\bar{c}_1+c))^{\frac{1}{d-\alpha}}\wedge \bar{\delta}_1$, where $c>0$. 
Then for every $|\bar{x}|<1/4$ and $w\in B(\bar{x},\delta_1)$, combining \eqref{formula}, \eqref{gornja} and \eqref{donja}, we obtain
\begin{align*}
 G_{B_{3/4}}(\bar{x},w)\geq \bar{c}_2\cdot |\bar{x}-w|^{\alpha-d}-\bar{c}_1\geq c>0.
\end{align*}
Define $c_2=c$ and now the statement follows.
\end{proof}

\begin{remark}\label{nap}
Notice that, according to the Lemma \ref{prva_gornja} 
and Lemma \ref{prva_donja_1}, there are $\tilde{c}=\tilde{c}(\alpha,d)$ and $\delta_1=\delta_1(\alpha,d)$ such that
for 
every $\bar{x}\in B_{1/4}$ and for every $w\in B(\bar{x},\delta_1)$
the inequality
\begin{align*}
  \int_{B_{1/2}}\ G_{B_{3/4}}(x,w)\ \mathrm{d}x \leq \tilde{c}\cdot G_{B_{3/4}}(\bar{x},w)
\end{align*}
holds.
\end{remark}

\begin{lemma}\label{treca_donja_1}
 For $\alpha \in (0,2)$ and $d\geq 2$, let  $\delta_1>0$ be as in Lemma \ref{prva_donja_1}. There is a constant $c_3=c_3(\alpha,d)$ such that for every $\bar{x}\in B_{1/4}$ and every $\tilde{u}\in B(0,3/4)\setminus B(\bar{x},\delta_1)$
 the inequality
\begin{align*}
\int_{B_{1/2}}\ G_{B_{3/4}} (x,\tilde{u})\ \mathrm{d}x \leq c_3 \cdot G_{B_{3/4}} (\bar{x}, \tilde{u})
\end{align*}
holds.
\end{lemma}

\begin{proof}
The proof relies on the maximum principle (cf. \cite{landkof}).
We use the 
fact that
$G_D (\bar{x}, \cdot)$ is regular harmonic in $D \setminus B(\bar{x}, \varepsilon)$ with respect to $X$  for every $\varepsilon > 0$ (cf. \cite{bg-1}). 

\begin{align*}
\int_{B_{1/2}} \ G_{B_{3/4}}(x,\tilde{u})\ \mathrm{d}x &= \int_{B_{1/2}} \mathbb{E}^{\tilde{u}}[G_{B_{3/4}}(x, X_{\tau_{{B_{3/4}}\setminus B(\bar{x},\delta_1)}})]\ \mathrm{d}x\\
							 &= \int_{B_{1/2}} \bigg[\ \int_{\mathbb{R}^2\setminus ({B_{3/4}}\setminus B(\bar{x},\delta_1))}\\
							 &\text{\;\;\;\;\;\;\;\;\;\;\;\;\;\;\;\;\;\;}G_{B_{3/4}}(x,z)\ P_{{B_{3/4}}\setminus B(\bar{x},\delta_1)}(\tilde{u},z)\ \mathrm{d}z\bigg]\ \mathrm{d}x\\
							 &= \int_{\mathbb{R}^2\setminus ({B_{3/4}}\setminus B(\bar{x},\delta_1))} \ \bigg[\int_{B_{1/2}} \\
							 &\text{\;\;\;\;\;\;\;\;\;\;\;\;\;\;\;\;\;\;}G_{B_{3/4}}(x,z)\ P_{{B_{3/4}}\setminus B(\bar{x},\delta_1)}(\tilde{u},z)\ \mathrm{d}x\bigg]\ \mathrm{d}z\\
							 &= \int_{B(\bar{x},\delta_1)}\bigg[\int_{B_{1/2}} \\
							 &\text{\;\;\;\;\;\;\;\;\;\;\;\;\;\;\;\;\;\;}G_{B_{3/4}}(x,z)\ P_{{B_{3/4}}\setminus B(\bar{x},\delta_1)}(\tilde{u},z)\ \mathrm{d}x\bigg]\ \mathrm{d}z\\
							 &\text{\;\;}+ \int_{\mathbb{R}^2\setminus B_{3/4}} \bigg[\int_{B_{1/2}} \\
							 &\text{\;\;\;\;\;\;\;\;\;\;\;\;\;\;\;\;\;\;}G_{B_{3/4}}(x,z)\ P_{{B_{3/4}}\setminus B(\bar{x},\delta_1)}(\tilde{u},z)\ \mathrm{d}x\bigg]\ \mathrm{d}z\\
							 &=\int_{B(\bar{x},\delta_1)}  \bigg[\int_{B_{1/2}}\\
							 &\text{\;\;\;\;\;\;\;\;\;\;\;\;\;\;\;\;\;\;}G_{B_{3/4}}(x,z)\ P_{{B_{3/4}}\setminus B(\bar{x},\delta_1)}(\tilde{u},z)\ \mathrm{d}x\bigg]\ \mathrm{d}z\\     
                                                         &\overset{Rem. \ref{nap}}{\leq}\tilde{c}\int_{B(\bar{x},\delta_1)} \ G_{B_{3/4}}(\bar{x},z)P_{{B_{3/4}}\setminus B(\bar{x},\delta_1)}(\tilde{u},z)\ \mathrm{d}z\\
							 &= \tilde{c} \int_{\mathbb{R}^2\setminus ({B_{3/4}}\setminus B(\bar{x},\delta_1))}
							 G_{B_{3/4}}(\bar{x},z)P_{{B_{3/4}}\setminus B(\bar{x},\delta_1)}(\tilde{u},z)\ \mathrm{d}z\\
							 &= \tilde{c} \cdot \mathbb{E}^{\tilde{u}}[G_{B_{3/4}}(\bar{x}, X_{\tau_{{B_{3/4}}\setminus B(\bar{x},\delta_1)}})]\\
							 &= \tilde{c} \cdot G_{B_{3/4}} (\bar{x}, \tilde{u}).
\end{align*}

Define $c_3=\tilde{c}$ and the lemma follows.
\end{proof}

\begin{proof}[Proof of Theorem \ref{Weak Harnack}]
\begin{align}
&\|u\|_{L^{1}(B_{1/2})}\notag\\
&=\int_{B_{1/2}} u(x)\ \mathrm{d}x\notag\\
&=\int_{B_{1/2}} \mathbb{E}^{x}[u(X_{\tau_{B_{3/4}}})]\ \mathrm{d}x\notag \\
&=\int_{B_{1/2}} \bigg[\int_{(\bar{B}_{3/4})^{c}} \ u(y)P_{B_{3/4}} (x,y)\ \mathrm{d}y\bigg]\mathrm{d}x\notag \\ 
&=\int_{B_{1/2}} \Bigg[\int_{(\bar{B}_{3/4})^{c}} u(y)\ \bigg[\int _{B(y,3/4)} f_\nu(z) \ G_{B_{3/4}}(x,y-z)\ \mathrm{d}z\bigg]\ \mathrm{d}y\Bigg]\ \mathrm{d}x\notag \\
&= \int_{(\bar{B}_{3/4})^{c}} u(y) \Bigg[\int _{B(y,3/4)} f_\nu(z)\ \bigg[\int_{B_{1/2}}  G_{B_{3/4}}(x,y-z)\ \mathrm{d}x\bigg]\ \mathrm{d}z\Bigg]\ \mathrm{d}y\label{zad}
\end{align}

By Remark \ref{nap} and Lemma \ref{treca_donja_1}
there is a constant $c=c(\alpha,d)$  
such that for every $\bar{x}\in B_{1/4}$
\begin{align*}  
  \eqref{zad}\ &\leq\ c\cdot \int_{(\bar{B}_{3/4})^{c}}  u(y)\bigg[\int _{B(y,3/4)} \mathbf{1}_{B(\bar{x},\delta_1)}(y-z)\\
  &\text{\;\;\;\;\;\;\;\;\;\;\;\;\;\;\;\;\;\;\;\;\;\;\;\;\;\;\;\;\;\;\;\;\;\;\;\;\;\;\;\;\;\;\;\;\;\;\;\;\;\;}f_\nu(z)\cdot G_{B_{3/4}}(\bar{x},y-z)\ \mathrm{d}z\bigg]\ \mathrm{d}y\\
  &\text{\;\;}+ c\cdot \int_{(\bar{B}_{3/4})^{c}} u(y)\bigg[\int _{B(y,3/4)} \mathbf{1}_{B(\bar{x},\delta_1)^c}(y-z)\\
  &\text{\;\;\;\;\;\;\;\;\;\;\;\;\;\;\;\;\;\;\;\;\;\;\;\;\;\;\;\;\;\;\;\;\;\;\;\;\;\;\;\;\;\;\;\;\;\;\;\;\;\;}f_\nu(z)\cdot G_{B_{3/4}}(\bar{x},y-z)\ \mathrm{d}z\bigg]\ \mathrm{d}y\\
  &= c\cdot \int_{(\bar{B}_{3/4})^{c}} u(y)\bigg[\int _{B(y,3/4)} f_\nu(z)\ G_{B_{3/4}}(\bar{x},y-z)\ \mathrm{d}z\bigg]\ \mathrm{d}y\\
  &=c\cdot \int_{(\bar{B}_{3/4})^{c}} u(y)P_{B_{3/4}} (\bar{x},y)\ \mathrm{d}y\\
  &=c\cdot \mathbb{E}^{\bar{x}}[u(X_{\tau_{{B_{3/4}}}})]\\
  &=c\cdot u(\bar{x}).
\end{align*}

Since the inequality
\begin{align*}
 \|u\|_{L^{1}(B_{1/2})}\leq c\cdot u(\bar{x}),
\end{align*}
holds for any $\bar{x}\in B_{1/4}$, the proof is finished.
\end{proof}

\hfill \break

\section{Appendix}
\label{Ap}

\begin{lemma}{\label{lema_pr}}
Let:
\begin{enumerate}
\item $\alpha_k = 2^{-k}, \ \beta_k = \frac{1}{k(k+1)}$, 
\item $\alpha_k = a^{-k}, \ \beta_k = b^{-k}$, \  $1 < b \leq a$,
\item $\alpha_k =  \frac{1}{k(k+1)} , \ \beta_k = k^{-1-\delta},\ \delta > 0$, 
\item $\alpha_k = 2^{-(k^3)}, \ \beta_k = 2^{-(k^2)}$\\
$\alpha_k = a^{-(k^3)}, \ \beta_k = b^{-(k^2)}$,\quad $a>1,b>1$.
\end{enumerate}
\end{lemma}
Then none of these sequences does not satisfy the property \eqref{uvjet}
from the Lemma \ref{lema}.

\begin{proof}
Notice that
\begin{align*}
\lim_{n \to \infty}\frac{\tan(A_n + B_{n+1})}{\tan(A_n + B_n)}=\infty
\end{align*}
implies
\begin{align}\label{limes1}
\lim_{n \to \infty}\frac{\cos(A_n + B_{n+1})}{\cos(A_n + B_n)}=0.
\end{align}

\begin{enumerate}
 \item Set
 \begin{align*}
 \alpha_k = 2^{-k}, \qquad \beta_k = \frac{1}{k(k+1)}.
 \end{align*}
 
 Since $A_n = 1-2^{-n}$, $B_n=\frac{n}{n+1}$ and $B_{n+1}=B_n+\beta_{n+1}$, we have:
 \begin{align*}
 \lim_{n \to \infty}\frac{\cos(A_n + B_{n+1})}{\cos(A_n + B_n)}&=\lim_{n \to \infty}\frac{\frac{d}{dn}(A_n+B_{n+1})}{\frac{d}{dn}(A_n+B_n)}\\
 &=1+\lim_{n \to \infty}\frac{\frac{d}{dn}\beta_{n+1}}{\frac{d}{dn}(A_n+B_n)}\\
 &=1+ \lim_{n \to \infty} \frac{\frac{-2n-3}{(n+1)^2(n+2)^2}}{\ln2\cdot 2^{-n}+(n+1)^{-2}}\\
 &=1+0=1.
 \end{align*}
 
This means that the condition (\ref{limes1}) is not fulfilled.

\item Set
 \begin{align*}
 \alpha_k = a^{-k}, \ \beta_k =  b^{-k},\quad 1<b\leq a.
 \end{align*}
 
 Since $A_n = \frac{1-a^{-n}}{a-1}$, $B_n=\frac{1-b^{-n}}{b-1}$, we have:
 \begin{align*}
 \lim_{n \to \infty}\frac{\cos(A_n + B_{n+1})}{\cos(A_n + B_n)}&=\lim_{n \to \infty}\frac{\frac{d}{dn}(A_n+B_{n+1})}{\frac{d}{dn}(A_n+B_n)}\\
 &=1+ \lim_{n \to \infty}\frac{\frac{d}{dn}\beta_{n+1}}{\frac{d}{dn}(A_n+B_n)}\\
 &=1+ \lim_{n \to \infty} \frac{-\ln b\cdot b^{-(n+1)}}{\frac{\ln a}{a-1}a^{-n}+ \frac{\ln b}{b-1}b^{-n}}\\
 &=1+ \lim_{n \to \infty} \frac{-\ln b\cdot b^{-1}}{\frac{\ln a}{a-1}(\frac{b}{a})^{n}+ \frac{\ln b}{b-1}}.
 \end{align*}
 
If $a=b$, the last expression is equal to

\begin{align*}
 1- \frac{\ln a\cdot a^{-1}}{2\frac{\ln a}{a-1}}&=1-\frac{a-1}{2a}\\
 &=\frac{a+1}{2a}.
\end{align*}

If $b<a$, we obtain

\begin{align*}
 1- \frac{\ln b\cdot b^{-1}}{\frac{\ln b}{b-1}}&=1-\frac{b-1}{b}\\
 &=\frac{1}{b}.
\end{align*}

In conclusion, the condition (\ref{limes1}) is not fulfilled.

\begin{remark}
In the examples that follow, we will proceed in the following way.

We use the equality
\begin{align*}
\frac{\cos(A_n + B_{n+1})}{\cos(A_n + B_n)}
=
\frac{\sin(\frac{\pi}{2}-A_n - B_{n+1})}{\sin(\frac{\pi}{2}-A_n - B_n)},
\end{align*}
and the estimates
\begin{align}
c_1\cdot \frac{\frac{\pi}{2}-A_n - B_{n+1}}{\frac{\pi}{2}-A_n - B_n}\leq\frac{\sin(\frac{\pi}{2}-A_n - B_{n+1})}{\sin(\frac{\pi}{2}-A_n - B_n)}\label{oc_1}\\
\frac{\sin(\frac{\pi}{2}-A_n - B_{n+1})}{\sin(\frac{\pi}{2}-A_n - B_n)}\leq c_2\cdot \frac{\frac{\pi}{2}-A_n - B_{n+1}}{\frac{\pi}{2}-A_n - B_n}\label{oc_2},
\end{align}
for some constants $c_1, c_2>0$.

Examining the lower (or upper) bound in the expression \eqref{oc_1} (or \eqref{oc_2}), we see
\begin{align}
\frac{\frac{\pi}{2}-A_n - B_n-\beta_{n+1}}{\frac{\pi}{2}-A_n - B_n}&=1-\frac{\beta_{n+1}}{\frac{\pi}{2}-A_n - B_n}\notag\\
&=1-\frac{\beta_{n+1}}{\sum_{k=n+1}^\infty \alpha_k + \sum_{k=n+1}^\infty \beta_k}\label{limes2}.
\end{align}

Now, we can use integral test for convergence to estimate series

$$
\sum_{k=n+1}^\infty \alpha_k
$$
and
$$
\sum_{k=n+1}^\infty \beta_k.
$$
\end{remark}
 
 \item Let 
\begin{align*}
 \alpha_k =  \frac{1}{k(k+1)} , \qquad \beta_k = k^{-1-\delta},
\end{align*}
 for some $\delta > 0$.
 In this example we will use integral test.
 
 Notice,
 \begin{align*}
 \sum_{k=n+1}^\infty \alpha_k = \frac{1}{n+1}. 
 \end{align*}
 
The integral test implies
 \begin{align*}
 \int_{n+1}^\infty \beta(x)dx \leq \sum_{k=n+1}^\infty \beta_k \leq \beta_{n+1} + \int_{n+1}^\infty \beta(x)dx,
 \end{align*}
 which is equivalent to
 \begin{align*}
  \sum_{k=n+1}^\infty \alpha_k + \int_{n+1}^\infty \beta(x)dx &\leq  \sum_{k=n+1}^\infty \alpha_k + \sum_{k=n+1}^\infty \beta_k,\notag\\
  \sum_{k=n+1}^\infty \alpha_k + \sum_{k=n+1}^\infty \beta_k &\leq \sum_{k=n+1}^\infty \alpha_k + \beta_{n+1} + \int_{n+1}^\infty \beta(x)dx.
 \end{align*}
 These inequalities yield
 \begin{align*}
  \frac{1}{\sum_{k=n+1}^\infty \alpha_k + \beta_{n+1} + \int_{n+1}^\infty \beta(x)dx} &\leq \frac{1}{\sum_{k=n+1}^\infty \alpha_k + \sum_{k=n+1}^\infty \beta_k},\notag\\
  \frac{1}{\sum_{k=n+1}^\infty \alpha_k + \sum_{k=n+1}^\infty \beta_k} &\leq \frac{1}{\sum_{k=n+1}^\infty \alpha_k + \int_{n+1}^\infty \beta(x)dx}.
 \end{align*}
 
 Multiplying the two inequalities by $\beta_{n+1}$, we obtain the lower and the upper bound for the expression in (\ref{limes2}),
 
 \begin{align*}
 \frac{\beta_{n+1}}{\sum_{k=n+1}^\infty \alpha_k + \beta_{n+1} + \int_{n+1}^\infty \beta(x)dx}
 \leq \frac{\beta_{n+1}}{\sum_{k=n+1}^\infty \alpha_k + \sum_{k=n+1}^\infty \beta_k}
 \end{align*}
 and
 
 \begin{align*}
 \frac{\beta_{n+1}}{\sum_{k=n+1}^\infty \alpha_k + \sum_{k=n+1}^\infty \beta_k}
  \leq \frac{\beta_{n+1}}{\sum_{k=n+1}^\infty \alpha_k + \int_{n+1}^\infty \beta(x)dx}.
 \end{align*}
 
 Since
 \begin{align*}
 \lim_{n \to \infty} \frac{\sum_{k=n+1}^\infty \alpha_k}{\beta_{n+1}}=\infty, 
 \end{align*}
 it follows
 
 \begin{align*}
 \lim_{n \to \infty} \frac{\beta_{n+1}}{\sum_{k=n+1}^\infty \alpha_k + \int_{n+1}^\infty \beta(x)dx}=0
 \end{align*}
 
 and consequently
 \begin{align*}
 \lim_{n \to \infty}\frac{\beta_{n+1}}{\sum_{k=n+1}^\infty \alpha_k + \sum_{k=n+1}^\infty \beta_k}=0.
 \end{align*}
 This means, by \eqref{limes2},
 
 \begin{align*}
 \lim_{n \to \infty} \frac{\frac{\pi}{2}-A_n - B_{n+1}}{\frac{\pi}{2}-A_n -
B_n}=1-0=1,
 \end{align*}
 which implies the condition \eqref{limes1} is not fulfilled.
 
 \item Set
  \begin{align*}
  \alpha_k = 2^{-(k^3)}, \ \ \beta_k = 2^{-(k^2)}.
  \end{align*}
 As in the second example, we use the integral test.
 Similarly, it implies
 \begin{align*}
  &\int_{n+1}^\infty \alpha(x)dx +  \int_{n+1}^\infty \beta(x)dx \leq \sum_{k=n+1}^\infty \alpha_k + \sum_{k=n+1}^\infty \beta_k,\\
  &\sum_{k=n+1}^\infty \alpha_k + \sum_{k=n+1}^\infty \beta_k \leq \alpha_{n+1}+\int_{n+1}^\infty \alpha(x)dx + \beta_{n+1}+\int_{n+1}^\infty \beta(x)dx,
 \end{align*}
 which yields
 \begin{align}
 &\frac{1}{\alpha_{n+1}+\int_{n+1}^\infty \alpha(x)dx + \beta_{n+1}+\int_{n+1}^\infty \beta(x)dx}\notag \\
 &\leq \frac{1}{\sum_{k=n+1}^\infty \alpha_k + \sum_{k=n+1}^\infty \beta_k},\notag\\
 &\frac{1}{\sum_{k=n+1}^\infty \alpha_k + \sum_{k=n+1}^\infty \beta_k} \leq \frac{1}{\int_{n+1}^\infty \alpha(x)dx +  \int_{n+1}^\infty \beta(x)dx}.\label{njd}
 \end{align}

 We will only observe the inequality \eqref{njd}.
 Multiplying (\ref{njd}) by $\beta_{n+1}$, it follows
 \begin{align*}
 \frac{\beta_{n+1}}{\sum_{k=n+1}^\infty \alpha_k + \sum_{k=n+1}^\infty \beta_k}
  \leq \frac{\beta_{n+1}}{\int_{n+1}^\infty \alpha(x)dx +  \int_{n+1}^\infty \beta(x)dx}.
 \end{align*}
 Since
 \begin{align*}
 \frac{\beta_{n+1}}{\int_{n+1}^\infty \alpha(x)dx +  \int_{n+1}^\infty \beta(x)dx} \leq \frac{\beta_{n+1}}{\int_{n+1}^\infty \beta(x)dx},
 \end{align*}
 it is enough to show
 \begin{equation*}\label{uvjet_pr}
 \lim_{n \to \infty}\frac{\beta_{n+1}}{\int_{n+1}^\infty \beta(x)dx}=0.
 \end{equation*}
 
 Notice that then 
 \begin{align*}
 \lim_{n \to \infty} \frac{\beta_{n+1}}{\sum_{k=n+1}^\infty \alpha_k + \sum_{k=n+1}^\infty \beta_k}=0,
 \end{align*}
 which implies the condition \eqref{limes1} is not fulfilled.
 
Using polar coordinates we obtain
 \begin{align*}
 \int_{n+1}^\infty \beta(x)dx &=\int_{n+1}^\infty 2^{-(x^2)}dx\\
 &=\frac{1}{2}\sqrt{\frac{\pi}{\ln2}}\cdot 2^{-\frac{1}{2}(n+1)^2},
 \end{align*}
 which implies
 \begin{align*}
 \lim_{n \to \infty} \frac{\beta_{n+1}}{\int_{n+1}^\infty \beta(x)dx}&=\lim_{n \to \infty}\frac{2^{-(n+1)^2}}{\frac{1}{2}\sqrt{\frac{\pi}{\ln2}}\cdot 2^{-\frac{1}{2}(n+1)^2}}\\
 &=\lim_{n \to \infty}\frac{2^{-\frac{1}{2}(n+1)^2}}{\frac{1}{2}\sqrt{\frac{\pi}{\ln2}}}\\
 &=0,
 \end{align*}
 and that is what we wanted to prove.
 \end{enumerate}
\end{proof}

\begin{remark}
 Notice that only in the very end the explicit formula for $\beta_n$ has been used. This means that if
 \begin{align*}
 \beta_k = b^{-(k^2)},\ b>1,
 \end{align*}
 the same reasoning applies.
 
 In this case, we would have
 \begin{align*}
 \int_{n+1}^\infty \beta(x)dx &=\int_{n+1}^\infty b^{-(x^2)}dx\\
 &=\frac{1}{2}\sqrt{\frac{\pi}{\ln b}}\cdot b^{-\frac{1}{2}(n+1)^2}
 \end{align*}
 and
  \begin{align*}
 \lim_{n \to \infty} \frac{\beta_{n+1}}{\int_{n+1}^\infty \beta(x)dx}&=\lim_{n \to \infty}\frac{b^{-(n+1)^2}}{\frac{1}{2}\sqrt{\frac{\pi}{\ln b}}\cdot b^{-\frac{1}{2}(n+1)^2}}\\
 &=\lim_{n \to \infty}\frac{b^{-\frac{1}{2}(n+1)^2}}{\frac{1}{2}\sqrt{\frac{\pi}{\ln b}}}\\
 &=0.
 \end{align*}
\end{remark}

\begin{lemma}{\label{lema4}}
Let the quantities $\delta_n$, $I_n$, $P_n$, $P_n'$, $\delta_n$, $x_n$, $y_n$, $c_n$ and $d_n$ be as in the Proposition \ref{prop4}. Then:
\begin{enumerate}
 \item $d_n=(\delta_n + x_n)\cdot \tan(A_n+B_{n+1})-1$;
 \item all quantities above are positive and $c_n<d_n$.
\end{enumerate}
\end{lemma}

\begin{proof}
\begin{enumerate}
\item
By the definition \eqref{def_del_n},
\begin{align*}
&\delta_n=\frac{I_n(1+\tan(A_n+B_n))\cdot(1+\tan(A_{n+1}+B_{n+1}))}{P_n^2(1+\tan(A_n+B_{n+1}))+P_n'I_n(1+\tan(A_{n+1}+B_{n+1}))}.
\end{align*}
This expression is equivalent to
\begin{align*}
 &I_n(1+\tan(A_n+B_n))\cdot(1+\tan(A_{n+1}+B_{n+1}))\\
 &=\delta_n P_n'I_n\cdot(1+\tan(A_{n+1}+B_{n+1}))+\delta_n P_n^2\cdot(1+\tan(A_n+B_{n+1})),
\end{align*}
from where
\begin{align*}
 &I_n(1+\tan(A_n+B_n))\cdot(1+\tan(A_{n+1}+B_{n+1}))\\
 &=\delta_n P_n'I_n\cdot(1+\tan(A_{n+1}+B_{n+1}))\\
 &+\delta_nP_n(1+\tan(A_n+B_{n+1}))\cdot(I_n-2(1+\tan(A_n+B_{n+1})))
\end{align*}
follows.
Therefore,
\begin{align*}
 &2\delta_n P_n\cdot(1+\tan(A_n+B_{n+1}))^2 \\
 &=\delta_n P_n I_n\cdot(1+\tan(A_n+B_{n+1}))\\
 &+I_n(1+\tan(A_{n+1}+B_{n+1}))\cdot (\delta_n P_n'-\tan(A_n+B_n)-1),
\end{align*}
and eventually
\begin{align*}
 &2\delta_n\frac{P_n}{I_n}\cdot(1+\tan(A_n+B_{n+1}))=\\
 &=\delta_n P_n +\frac{1+\tan(A_{n+1}+B_{n+1})}{1+\tan(A_n+B_{n+1})}\cdot (\delta_n P_n'-\tan(A_n+B_n)-1).
\end{align*}
By the definition of $x_n$ \eqref{def_x_n}, the last equality yields
\begin{align*}
&x_n\cdot\tan(A_n+B_{n+1})-1=\\
 &=\delta_n P_n +\frac{1+\tan(A_{n+1}+B_{n+1})}{1+\tan(A_n+B_{n+1})}\cdot (\delta_nP_n'-\tan(A_n+B_n)-1).
\end{align*}
Therefore,
\begin{align*}
 \frac{1+\tan(A_n+B_n)-\delta_n P_n'}{1+\tan(A_n+B_{n+1})}=\frac{\delta_n P_n-x_n\cdot\tan(A_n+B_{n+1})+1}{1+\tan(A_{n+1}+B_{n+1})}.
\end{align*}
Notice that by \eqref{def_y_n} the left-hand side is $y_n$, hence
\begin{align*}
 y_n=\frac{\delta_n P_n-x_n\cdot\tan(A_n+B_{n+1})+1}{1+\tan(A_{n+1}+B_{n+1})}.
\end{align*}
From here, by the definition of $P_n$ \eqref{def_P_n},
\begin{align*}
&(\delta_n-y_n)\cdot\tan(A_{n+1}+B_{n+1})-y_n\\
&=(\delta_n + x_n)\cdot \tan(A_n+B_{n+1})-1.
\end{align*}
Consequently, by the definition of $d_n$ \eqref{def_d_n}, the last equality implies
 \begin{align*}
 d_n=(\delta_n + x_n)\cdot \tan(A_n+B_{n+1})-1,
 \end{align*}
and hence the first part of the Lemma.

\item
Using the definition of $\delta_n$ \eqref{def_del_n} one sees that $y_n>0$. Namely, the nominator of $y_n$ is positive if and only if 
\begin{align*}
 1+\tan(A_n+B_n)-\delta_n\cdot P_n'>0,
\end{align*}
which is equivalent to
\begin{align*}
 \frac{1+\tan(A_n+B_n)}{P_n'}>\delta_n.
\end{align*}
By the definition of $\delta_n$, the condition above is fulfilled if and only if
\begin{align*}
  &\frac{1+\tan(A_n+B_n)}{P_n'}\\
  &>\frac{I_n\cdot (1+\tan(A_n+B_n))\cdot (1+\tan(A_{n+1}+B_{n+1}))}{H_n},
\end{align*}
which is equivalent to
\begin{align*}
 H_n>P_n'\cdot I_n\cdot (1+\tan(A_{n+1}+B_{n+1})).
\end{align*}
Using \eqref{def_H_n}, we see that it holds.

Consequently, $c_n$ is also positive.

Next we show that $c_n<d_n$, which also implies that $d_n>0$.
Since
\begin{align*}
 I_n>0,
\end{align*}
by the definition of $P_n$ \eqref{def_P_n}, the inequality
\begin{align*}
 2\cdot (1+\tan(A_{n+1}+B_{n+1}))>P_n
\end{align*}
holds.

Therefore,
\begin{align*}
&2\cdot P_n (1+\tan(A_n+B_{n+1}))\cdot(1+\tan(A_{n+1}+B_{n+1}))\\
&>P_n^2(1+\tan(A_n+B_{n+1})).
\end{align*}

Using the expression for $H_n$ \eqref{def_H_n},
\begin{align*}
 &2\cdot P_n (1+\tan(A_n+B_{n+1}))\cdot(1+\tan(A_{n+1}+B_{n+1}))\\
 &+P_n'I_n\cdot(1+\tan(A_{n+1}+B_{n+1}))> H_n,
\end{align*}
which, by the definition of $\delta_n$ \eqref{def_del_n} yields
\begin{align*}
 \delta_n>\frac{I_n\cdot(1+\tan(A_n+B_n))}{2\cdot P_n (1+\tan(A_n+B_{n+1}))+P_n'I_n}.
\end{align*}

Therefore,
\begin{align*}
 2\delta_n\cdot P_n (1+\tan(A_n+B_{n+1}))> (1+\tan(A_n+B_n)-\delta_n P_n')\cdot I_n,
\end{align*}
which implies
\begin{align*}
 2\delta_n\cdot\frac{1+\tan(A_n+B_{n+1})}{\tan(A_n+B_{n+1})}\cdot \frac{P_n}{I_n}>\frac{1+\tan(A_n+B_n)-\delta_n P_n'}{\tan(A_n+B_{n+1})}.
\end{align*}

Consequently, by the definition of $x_n$ \eqref{def_x_n} 
\begin{align*}
x_n>\frac{2+\tan(A_n+B_n)-\delta_n P_n'}{\tan(A_n+B_{n+1})}.
\end{align*}

Hence, using the definition for $y_n$ \eqref{def_y_n}, 
\begin{align*}
 \frac{x_n\cdot\tan(A_n+B_{n+1})-1}{1+\tan(A_n+B_{n+1})}>y_n.
\end{align*}

Therefore,
\begin{align*}
(\delta_n + x_n)\cdot \tan(A_n+B_{n+1})-1> (\delta_n + y_n)\cdot\tan(A_n+B_{n+1})+y_n.
\end{align*}

By the first part of the Lemma and the definition of $c_n$ \eqref{def_c_n}, the equality
\begin{align*}
d_n>c_n,
\end{align*}
holds.

Hence $d_n>0$ and the proof is finished.
\end{enumerate}

\end{proof}

\begin{lemma}{\label{lema5}}
Let $I_n$, $J_n$, $P_n$, $P_n'$, $\delta_n$, $x_n$, $y_n$, $c_n$, $d_n$, $U_n$, $\delta_n'$, $S_n=(x_{S_n}, y_{S_n})$ and $r_n$ be as in the Proposition \ref{prop5}. Then:
\begin{enumerate}
 \item $B(S_n, r_n)\subset Q_n:=(\delta_n', \delta_n)\times (c_n,c_n+(\delta_n-\delta_n'))$;
 \item the inequalities
\begin{align*}
 c_n &\geq(\varepsilon+1)\cdot\tan(A_n+B_n)+1,\\
 c_n+(\delta_n-\delta_n') &\leq(\varepsilon+x_n)\cdot\tan(A_n+B_{n+1})-1,
\end{align*}
 hold true, for some $\varepsilon>0$ such that $\delta_n'<\varepsilon<\delta_n$.
\end{enumerate}
\end{lemma}

\begin{proof}
\begin{enumerate}
\item Let $z=(z_1, z_2)\in B(S_n, r_n)$. Clearly,
\begin{align*}
|z_1-x_{S_n}|<r_n,\\
|z_2-y_{S_n}|<r_n.
\end{align*}

If we show:
\begin{align}\label{prvo}
x_{S_n}-r_n=\delta_n'
\end{align}
\begin{align}\label{drugo}
x_{S_n}+r_n=\delta_n
\end{align}
\begin{align}\label{trece}
y_{S_n}-r_n=c_n
\end{align}
\begin{align}\label{cetvrto}
y_{S_n}+r_n=c_n+(\delta_n-\delta_n'),
\end{align}
then the proof of the first part of the Lemma is finished.

Let us first check \eqref{prvo}. The definitions of $x_{S_n}$ \eqref{def_x_S}, $r_n$ \eqref{def_r_n}, $I_n$ \eqref{def_I_n}, $P_n$ \eqref{def_P_n} and $\delta_n'$ \eqref{def_del_n_c} imply
\begin{align*}
 x_{S_n}-r_n&=\frac{\delta_n+y_n}{2}\cdot\frac{I_n}{1+\tan(A_{n+1}+B_{n+1})}\\
&\text{\;\;\;\;}-\frac{1}{2}\cdot\frac{\delta_n  P_n-y_n  I_n}{1+\tan(A_{n+1}+B_{n+1})}\\
 &=\frac{\delta_n(I_n-P_n)+2y_nI_n}{2(1+\tan(A_{n+1}+B_{n+1}))}\\
 &=\frac{\delta_n(1+\tan(A_n+B_{n+1}))+y_nI_n}{1+\tan(A_{n+1}+B_{n+1})}\\
 &=\delta_n'.
\end{align*}
Let us show \eqref{drugo}. Similarly, the definitions of $x_{S_n}$ \eqref{def_x_S}, $r_n$ \eqref{def_r_n}, $I_n$ \eqref{def_I_n}, $P_n$ \eqref{def_P_n} and $\delta_n$ \eqref{def_del_n} imply
\begin{align*}
 x_{S_n}+r_n&=\frac{\delta_n+y_n}{2}\cdot\frac{I_n}{1+\tan(A_{n+1}+B_{n+1})}\\
&\text{\;\;\;\;}+\frac{1}{2}\cdot\frac{\delta_n  P_n-y_n  I_n}{1+\tan(A_{n+1}+B_{n+1})}\\
 &=\frac{\delta_n(I_n+P_n)}{2(1+\tan(A_{n+1}+B_{n+1}))}\\
 &=\frac{2\delta_n(1+\tan(A_{n+1}+B_{n+1}))}{2(1+\tan(A_{n+1}+B_{n+1}))}\\
 &=\delta_n.
\end{align*}
Now, we check \eqref{trece}. Likewise, the definitions of $y_{S_n}$ \eqref{def_y_S}, $r_n$ \eqref{def_r_n}, $I_n$ \eqref{def_I_n}, $J_n$ \eqref{def_J_n}, $P_n$ \eqref{def_P_n} and $c_n$ \eqref{def_c_n}  yield
\begin{align*}
 y_{S_n}-r_n&=\frac{\delta_n+y_n}{2}\cdot\frac{J_n}{1+\tan(A_{n+1}+B_{n+1})}\\
&\text{\;\;\;\;}-\frac{1}{2}\cdot\frac{\delta_n P_n-y_n I_n}{1+\tan(A_{n+1}+B_{n+1})}\\
 &=\frac{\delta_n(J_n-P_n)+y_n(I_n+J_n)}{2(1+\tan(A_{n+1}+B_{n+1}))}\\
 &=\delta_n\tan(A_n+B_{n+1})+y_n\tan(A_n+B_{n+1})+y_n\\
 &=(\delta_n+y_n)\cdot\tan(A_n+B_{n+1})+y_n\\
 &=c_n.
\end{align*}
Lastly, we check \eqref{cetvrto}. The definitions of $y_{S_n}$ \eqref{def_y_S}, $r_n$ \eqref{def_r_n}, $I_n$ \eqref{def_I_n}, $J_n$ \eqref{def_J_n} and $P_n$ \eqref{def_P_n} imply
\begin{align*}
 y_{S_n}+r_n&=\frac{\delta_n+y_n}{2}\cdot\frac{J_n}{1+\tan(A_{n+1}+B_{n+1})}\\
&\text{\;\;\;\;}+\frac{1}{2}\cdot\frac{\delta_n\cdot P_n-y_n\cdot I_n}{1+\tan(A_{n+1}+B_{n+1})}\\
 &=\frac{\delta_n\tan(A_{n+1}+B_{n+1})(1+\tan(A_n+B_{n+1}))}{1+\tan(A_{n+1}+B_{n+1})}\\
 &+\frac{y_n(\tan(A_n+B_{n+1})\tan(A_{n+1}+B_{n+1})-1)}{1+\tan(A_{n+1}+B_{n+1})}\\
 &=\frac{\delta_n\tan(A_{n+1}+B_{n+1})(1+\tan(A_n+B_{n+1}))}{1+\tan(A_{n+1}+B_{n+1})}\\
 &+\frac{y_n((1+\tan(A_n+B_{n+1}))(1+\tan(A_{n+1}+B_{n+1}))-I_n)}{1+\tan(A_{n+1}+B_{n+1})}\\
 &=\frac{(\delta_n+y_n)(1+\tan(A_n+B_{n+1}))(1+\tan(A_{n+1}+B_{n+1}))}{1+\tan(A_{n+1}+B_{n+1})}\\
 &+\frac{-\delta_n(1+\tan(A_n+B_{n+1}))-y_nI_n}{1+\tan(A_{n+1}+B_{n+1})}\\
 &=(\delta_n+y_n)(1+\tan(A_n+B_{n+1}))\\
 &+\frac{-\delta_n(1+\tan(A_n+B_{n+1}))-y_nI_n}{1+\tan(A_{n+1}+B_{n+1})}\\
 &=c_n+(\delta_n-\delta_n'),
\end{align*}
where in the last equality we use the definitions of $c_n$ \eqref{def_c_n} and $\delta_n'$ \eqref{def_del_n_c}. 
\item Recall,
\begin{align*}
\Delta_{n}^{\varepsilon}=\{(y_1,y_2): &y_1=\varepsilon,
(\varepsilon+1)\cdot\tan(A_n+B_n)+1<\\
&<y_2<(\varepsilon+x_n)\cdot\tan(A_n+B_{n+1})-1\}.
\end{align*}

We start by proving the inequality
\begin{align*}
(\varepsilon+1)\cdot\tan(A_n+B_n)+1\leq c_n,
\end{align*}
for some $\delta_n'<\varepsilon<\delta_n$.

Let $\delta_n'<\varepsilon<\delta_n$. 
Then the inequality
\begin{align*}
(\varepsilon+1)\cdot\tan(A_n+B_n)+1\leq (\delta_n+1) \tan(A_n+B_n)+1
\end{align*}
holds.
Therefore, the definition of $P_n'$ \eqref{def_P_n'} yields
\begin{align*}
&(\varepsilon+1)\cdot\tan(A_n+B_n)+1\\
&\leq \delta_n\cdot\tan(A_n+B_{n+1})+(1+\tan(A_n+B_n)-\delta_n P_n').
\end{align*}

The definition of $y_n$ \eqref{def_y_n} implies
\begin{align*}
(\varepsilon+1)\cdot\tan(A_n+B_n)+1\leq (\delta_n + y_n)\cdot\tan(A_n+B_{n+1})+y_n.
\end{align*}

Eventually, the definition of $c_n$ \eqref{def_c_n} implies
\begin{align*}
(\varepsilon+1)\cdot\tan(A_n+B_n)+1\leq c_n,
\end{align*}
and hence the inequality.

In order to prove the inequality
\begin{align*}
c_n+(\delta_n-\delta_n')\leq(\varepsilon+x_n)\cdot\tan(A_n+B_{n+1})-1,
\end{align*}
it is enough to show it for $\varepsilon=\delta_n'$.

We also use two relations:
\begin{align}
 &\delta_n-\delta_n'=\delta_n'\cdot\tan(A_{n+1}+B_{n+1})-\delta_n\cdot\tan(A_n+B_{n+1})-y_n\cdot I_n,\label{del}\\
 &y_n=\frac{\delta_n P_n-x_n\tan(A_n+B_{n+1})+1}{1+\tan(A_{n+1}+B_{n+1})}.\label{ips}
\end{align}
Notice that these statements follow directly from the definition of $\delta_n'$ \eqref{def_del_n_c} and the equality
\begin{align*}
d_n=(\delta_n+x_n)\cdot\tan(A_n+B_{n+1})-1
\end{align*}
from the Lemma \ref{lema4}, respectively.

Due to
\begin{align*}
&\delta_n'<\delta_n,
\end{align*}
and the definition of $P_n$ \eqref{def_P_n}, the inequality
\begin{align*}
&-\delta_n P_n + x_n\cdot\tan(A_n+B_{n+1})-1+\delta_n'\cdot\tan(A_{n+1}+B_{n+1})\\
& \leq(\delta_n'+x_n)\cdot\tan(A_n+B_{n+1})-1
\end{align*}
holds.
Therefore, the equality \eqref{ips} implies
\begin{align*}
&\delta_n'\cdot\tan(A_{n+1}+B_{n+1})-y_n\cdot(1+\tan(A_{n+1}+B_{n+1}))\\
 &\leq (\delta_n'+x_n)\cdot\tan(A_n+B_{n+1})-1.
\end{align*}
From here, using \eqref{del}, the inequality
\begin{align*}
&(\delta_n + y_n)\cdot\tan(A_n+B_{n+1})+y_n+(\delta_n-\delta_n')\\
&\leq(\delta_n'+x_n)\cdot\tan(A_n+B_{n+1})-1
\end{align*}
follows.
The definition of $c_n$ \eqref{def_c_n} yields,
\begin{align*}
&(\delta_n + y_n)\cdot\tan(A_n+B_{n+1})+y_n+(\delta_n-\delta_n')\\
 &\leq(\delta_n'+x_n)\cdot\tan(A_n+B_{n+1})-1,
\end{align*}
and the inequality is proved.

\end{enumerate}
\end{proof}

\begin{lemma}\label{har1}
Let $I_n$, $P_n$, $P_n'$, $\delta_n$, $x_n$, $y_n$, $c_n$, $d_n$, $U_n$, $\delta_n'$, $S_n=(x_{S_n}, y_{S_n})$, $r_n$, $y_{P_1}$ and $ y_{P_2}$ be as in the Proposition \ref{prop5}. Then
\begin{enumerate}
 \item $y_{P_1}\leq c_n$,
 \item $c_n + (\delta_n-\delta_n')\leq y_{P_2}$.
\end{enumerate}
\end{lemma}
\begin{proof}

\begin{enumerate}
\item The inequality
\begin{align*}
\varepsilon<\delta_n,
\end{align*}
implies
\begin{align*}
(\varepsilon + y_n)\cdot \tan(A_n+B_{n+1})+y_n \leq (\delta_n + y_n)\cdot\tan(A_n+B_{n+1})+y_n.
\end{align*}
Therefore, together with \eqref{y_P_1} and \eqref{def_c_n}, the inequality
\begin{align*}
y_{P_1}\leq c_n
\end{align*}
follows.

\item To prove the second inequality, we use $\delta_n'<\varepsilon$ and \eqref{del}.

Namely,
\begin{align*}
\delta_n'\leq \varepsilon
\end{align*}
implies
\begin{align*}
&\delta_n'\tan(A_{n+1}+B_{n+1}))-y_n(1+\tan(A_{n+1}+B_{n+1}))\\
&\leq \varepsilon\cdot\tan(A_{n+1}+B_{n+1})-y_n(1+\tan(A_{n+1}+B_{n+1})).
\end{align*}
From here we obtain
\begin{align*}
&(\delta_n + y_n)\cdot\tan(A_n+B_{n+1})+y_n +\\
&+\delta_n'\cdot\tan(A_{n+1}+B_{n+1})-\delta_n\cdot \tan(A_n+B_{n+1})-y_n\cdot I_n\leq\\
&\leq (\varepsilon-y_n)\cdot\tan(A_{n+1}+B_{n+1})-y_n
\end{align*}
Using \eqref{y_P_2}, \eqref{def_c_n} and \eqref{del} it follows
\begin{align*}
c_n + (\delta_n-\delta_n')\leq y_{P_2},
\end{align*}
and hence the Lemma.

\end{enumerate}
\end{proof}

\begin{lemma}\label{har}
Let $I_n$, $P_n$, $P_n'$, $\delta_n$, $y_n$, $b_n$ and $M$ be as in the proof of Theorem \ref{har_thm_mat}, i.e. from Proposition \ref{prop4}.
Then:
\begin{align*}
\limsup_{n\rightarrow \infty} \frac{M}{y_n}<\infty.
\end{align*}
\end{lemma}
\begin{proof}
By the definition of $\delta_n$ \eqref{def_del_n}, it follows:
\begin{align}
 y_n&=\frac{1+\tan(A_n+B_n)-\delta_n\cdot P_n'}{1+\tan(A_n+B_{n+1})}\notag\\
 &=\frac{P_n^2\cdot(1+\tan(A_n+B_n))}{H_n},\label{y_n_g}
\end{align}
where $H_n$ is like in the proof of the Proposition \ref{prop4}.

From here and \eqref{def_del_n} it follows
\begin{align}\label{komb1}
 &\frac{\delta_n}{y_n}=\frac{ I_n\cdot(1+\tan(A_{n+1}+B_{n+1}))}{P_n^2}.
\end{align}

Also, by \eqref{y_n_g} and the definitions of $\delta_n$ \eqref{def_del_n} and $x_n$ \eqref{def_x_n}, it follows
\begin{align}\label{komb2}
 \frac{x_n}{y_n}=R_1+R_2+R_3,
\end{align}
 where
\begin{align*}
 R_1&=2\cdot\frac{1+\tan(A_n+B_{n+1})}{\tan(A_n+B_{n+1})}\cdot\frac{1+\tan(A_{n+1}+B_{n+1})}{P_n},\\
 R_2&=\frac{1+\tan(A_n+B_{n+1})}{(1+\tan(A_n+B_n))\cdot\tan(A_n+B_{n+1})},\\
 R_3&=\frac{I_n\cdot(1+\tan(A_{n+1}+B_{n+1}))}{P_n^2} \frac{P_n'}{\tan(A_n+B_{n+1})}\frac{1}{1+\tan(A_n+B_n)}.
\end{align*}

Combining (\ref{komb1}) and (\ref{komb2}), the conditions (\ref{uvjet}) and (\ref{uvjet_d}) imply
\begin{align*}
 \limsup_{n\rightarrow \infty} \frac{M}{y_n}<\infty,
\end{align*}
and hence the proposition.
\end{proof}

\hfill \break

\bibliographystyle{amsplain}

\end{document}